\documentclass[11pt, reqno]{amsart}

\usepackage[T1]{fontenc}
\usepackage{xcolor}
\usepackage{latexsym}
\usepackage{amssymb}
\usepackage{mathrsfs}
\usepackage{amsmath}
\usepackage{fancybox,color}
\usepackage{enumerate}
\usepackage[latin1]{inputenc}
\usepackage{hyperref}
\usepackage{eurosym}
\usepackage{url}

\newcommand{\kom}[1]{}
\renewcommand{\kom}[1]{{\bf [#1]}}

\newcommand\blfootnote[1]{%
  \begingroup
  \renewcommand\thefootnote{}\footnote{#1}%
  \addtocounter{footnote}{-1}%
  \endgroup
}

 \def\1{\raisebox{2pt}{\rm{$\chi$}}}

\newtheorem{theorem}{Theorem}[section]
\newtheorem{corollary}[theorem]{Corollary}
\newtheorem{lemma}[theorem]{Lemma}
\newtheorem{proposition}[theorem]{Proposition}
\newtheorem{definition}[theorem]{Definition}
\newtheorem{remark}[theorem]{Remark}

 \def\1{\raisebox{2pt}{\rm{$\chi$}}}
 
\def\vint_#1{\mathchoice%
          {\mathop{\kern 0.2em\vrule width 0.6em height 0.69678ex depth -0.58065ex
                  \kern -0.8em \intop}\nolimits_{\kern -0.4em#1}}%
          {\mathop{\kern 0.1em\vrule width 0.5em height 0.69678ex depth -0.60387ex
                  \kern -0.6em \intop}\nolimits_{#1}}%
          {\mathop{\kern 0.1em\vrule width 0.5em height 0.69678ex
              depth -0.60387ex
                  \kern -0.6em \intop}\nolimits_{#1}}%
          {\mathop{\kern 0.1em\vrule width 0.5em height 0.69678ex depth -0.60387ex
                  \kern -0.6em \intop}\nolimits_{#1}}}
\def\vintslides_#1{\mathchoice%
          {\mathop{\kern 0.1em\vrule width 0.5em height 0.697ex depth -0.581ex
                  \kern -0.6em \intop}\nolimits_{\kern -0.4em#1}}%
          {\mathop{\kern 0.1em\vrule width 0.3em height 0.697ex depth -0.604ex
                  \kern -0.4em \intop}\nolimits_{#1}}%
          {\mathop{\kern 0.1em\vrule width 0.3em height 0.697ex depth -0.604ex
                  \kern -0.4em \intop}\nolimits_{#1}}%
          {\mathop{\kern 0.1em\vrule width 0.3em height 0.697ex depth -0.604ex
                  \kern -0.4em \intop}\nolimits_{#1}}}

\newcommand{\kint}{\vint}

\newcommand{\aveint}[2]{\mathchoice%
          {\mathop{\kern 0.2em\vrule width 0.6em height 0.69678ex depth -0.58065ex
                  \kern -0.8em \intop}\nolimits_{\kern -0.45em#1}^{#2}}%
          {\mathop{\kern 0.1em\vrule width 0.5em height 0.69678ex depth -0.60387ex
                  \kern -0.6em \intop}\nolimits_{#1}^{#2}}%
          {\mathop{\kern 0.1em\vrule width 0.5em height 0.69678ex depth -0.60387ex
                  \kern -0.6em \intop}\nolimits_{#1}^{#2}}%
          {\mathop{\kern 0.1em\vrule width 0.5em height 0.69678ex depth -0.60387ex
                  \kern -0.6em \intop}\nolimits_{#1}^{#2}}}

\newcommand{\esssup}{\operatornamewithlimits{esssup}}

\newcommand{\diam}{\operatorname{diam}}

\numberwithin{equation}{section}
\allowdisplaybreaks

\definecolor{color1}{rgb}{0.309, 0.43,0.258}
\definecolor{color2}{rgb}{0.741, 0.502,0.743}
\definecolor{color3}{rgb}{0.580, 0.163,0.107}
\definecolor{color0}{rgb}{0.088, 0.717, 0.018}

\begin{document}

\title[]{$L^{p}$-estimates for the Hessians of solutions to fully nonlinear parabolic equations with oblique boundary conditions}

\author[Byun]{Sun-Sig Byun}
\address{Department of Mathematical Sciences, Seoul National University 
1, Gwanak-ro, Gwanak-gu, Seoul, Republic of Korea}
\email{byun@snu.ac.kr}

\author[Han]{Jeongmin Han}
\address{Department of Mathematical Sciences, Seoul National University 
1, Gwanak-ro, Gwanak-gu, Seoul, Republic of Korea}
\email{hanjm9114@snu.ac.kr}

\blfootnote{S.-S. Byun was supported by NRF-2017R1A2B2003877. J. Han was supported by NRF-2019R1C1C1003844.}


\begin{abstract}
We study  fully nonlinear parabolic equations in nondivergence form with oblique boundary conditions. 
An optimal and global Calder\'{o}n-Zygmund estimate is obtained by proving that the Hessian of the viscosity solution  
to the oblique boundary problem is as integrable as the nonhomogeneous term in $L^{p}$ spaces 
under minimal regularity requirement on the nonlinear operator, the boundary data and the boundary of the domain.

\end{abstract}

\maketitle

\section{Introduction} 
In this paper, we study the following parabolic oblique boundary value problem
\begin{align} \label{ob_eqoripa}
\left\{ \begin{array}{ll}
F(D^{2}u, Du, u,  x, t) - u_{t} = f & \textrm{in $\Omega \times (0,T)$,}\\
\beta \cdot Du = 0 & \textrm{on $\partial \Omega \times (0,T)$,}\\
u(\cdot , 0) = 0 & \textrm{in $\Omega$,}
\end{array} \right. 
\end{align}
where  $\Omega \subset \mathbb{R}^{n}$ is a bounded domain with $n \ge 2$ and 
$T>0 $.  
Here, 
$ F:S(n) \times \mathbb{R}^{n} \times \mathbb{R} \times \Omega \times (0,T) \to  \mathbb{R}  $ is uniformly elliptic with constants $\lambda $ and $\Lambda $,
i.e. 
$$ \lambda ||N|| \le F(M+N,q,r,x,t)-F(M,q,r,x,t) \le \Lambda ||N||$$
for any $ M, N \in S(n)$, $N \ge 0 $, $q \in \mathbb{R}^{n}$, $r \in \mathbb{R} $ and $(x,t) \in \Omega \times (0,T)$,
and $ \beta :\partial \Omega \times (0,T) \to \mathbb{R}^{n}$ is a given function with $||\beta||_{L^{\infty}(\partial \Omega \times (0,T)) } \le 1$.
Under regularity assumptions on $\Omega,F$ and $\beta$, we prove global $W^{2,p}$-estimates for viscosity solutions of \eqref{ob_eqoripa}. 

This work is a natural outgrowth of the recent paper \cite{MR4046185} where $W^{2,p}$-regularity is proved for the fully nonlinear elliptic problem with oblique boundary conditions.
We point out that oblique boundary conditions give some information
but not specific value of the solution $u$ at the lateral boundary $\partial \Omega \times (0,T)$.
Thus, we have to solve the equation not only inside, but also at the boundary of the domain $ \Omega \times (0,T)$.
This is the main difference between our problem and the Dirichlet problem. 

To sketch the proof, first we need to establish $W^{2,p}$-regularity theory in a domain with flat boundary.
At the beginning, we first obtain a $W^{2,p}$-estimate in the case of $F=F(X,x,t)$ with a small BMO assumption.
And next we generalize this result to the case $F=F(X,q,r,x,t)$. 
Here we only consider $W^{1,p}$-regularity for $u$ since this regularity and the structure condition \eqref{paob_sc} for $F$ yield the desired regularity result in the flat domain. 
Then we can get the desired results by employing boundary regularity results of Dirichlet problems.
Finally, we can obtain global regularity results by using flattening and covering argument.
Throughout our proof, regularity results for limiting PDE \eqref{paromoeq}, which can be found in \cite{cm2019oblique}, play an essential role.
We compare our problem \eqref{paob_eq} with \eqref{paromoeq} that enjoys $C^{1,1}$-regularity up to the boundary as in \cite{cm2019oblique}, by employing the small BMO assumption in order to derive the desired comparison estimates.
In the process, we need to impose our boundary condition that $\partial \Omega \in C^{3} $, ensuring that $\beta \in C^{2}$. (See Remark \ref{rembeta})

The notion of viscosity solutions introduced by Evans \cite{MR597451} and named by Lions and Crandall \cite{MR690039} presents a new paradigm for studying PDEs.
In particular, this has led to a significant progress in researches on PDEs in nondivergence form.
The fundamental theory of viscosity solutions for fully nonlinear nondivergence equations can be found in \cite{MR1005611, MR1069735, MR1118699, MR1329831}, particularly, \cite{MR1135923,MR1139064,MR1151267} in the parabolic case.
Indeed, there have been many various results regarding nondivergence equations, like as in \cite{MR2735102,MR2738329,MR3013294,MR3124896,MR3169795,MR3437597}, etc.

Researches on the oblique boundary value problems have been actively conducted over the last four decades.  
Accordingly, regularity theory for these problems has also been developed.
In elliptic case, we refer to Lieberman's works \cite{MR923448, MR953664} for linear problems.
For nonlinear equations,  \cite{MR833695,MR1335754,MR1994804} have discovered regularity theory. 
Meanwhile, in parabolic case, there are several results such as \cite{MR765964,MR1046706,MR1111473,MR1192119}.
Besides, various results for oblique boundary problems can be found in articles such as \cite{MR1388600,MR1650200,MR1796316,MR2070626}, etc. 
On the other hand, to derive $W^{2,p}$-estimates for oblique problems,
it is essential to look at the regularity for limiting problems.
Milakis and Silvestre derived $C^{2,\alpha}$-regularity for elliptic problems with Neumann boundary data in \cite{MR2254613},
and Li and Zhang extended this to the general oblique problems in \cite{MR3780142}.
For parabolic problems, Chatzigeorgiou and Milakis have proved such regularity in \cite{cm2019oblique}.

The rest of the paper is organized as follows.
We introduce notations and some background about Hessian estimates for viscosity solutions of parabolic equations.
In Section 3, we collect several lemmas employed in the proof of our main result, Theorem \ref{paob_main}.
Finally, we derive the desired $W^{2,p}$-estimates for \eqref{ob_eqoripa} in Sections 4 and 5.

\section{Preliminaries}
\subsection{Notations}
We introduce some notations which will be used throughout this paper.
\begin{itemize} 
\item $x=(x_{1}, x_{2}, \dots, x_{n} )= (x',x_{n})  \in \mathbb{R}^{n}$, $(x,t) = (x',x_{n},t)$ for $t \in \mathbb{R}$. 
\item $\mathbb{R}_{+}^{n} : = \{ x \in \mathbb{R}^{n} : x_{n} > 0 \} $.
\item $B_{r}(x_{0}):=\{ x \in  \mathbb{R}^{n} :|x-x_{0}|<r \}  $ for $x_{0} \in  \mathbb{R}^{n} $, $r>0$, and $B_{r}^{+}(x_{0}):=B_{r}(x_{0}) \cap \mathbb{R}_{+}^{n}$. $B_{r}=B_{r}(0)$ and $B_{r}^{+}=B_{r}^{+}(0)$. 
\item $ B_{r,h} =B_{r}(-(R-h)e_{n})$, where $R$ satisfies $(R-h)^{2}+r^{2}=R^{2}$. 
$ B_{r,h}^{+}= B_{r,h}  \cap \mathbb{R}^{+}$.
\item $ T_{r}(x_{0}) :=   \{ x_{0}+x : x \in B_{r} \cap (\mathbb{R}^{n-1} \times \{ 0 \} ) \}$, $T_{r}= T_{r}(0) $.
\item $Q_{r}(x_{0},t_{0}):= B_{r}(x_{0}) \times (t_{0}-r^{2},t_{0})$ and $Q_{r}^{+}(x_{0},t_{0}):=B_{r}^{+}(x_{0})  \times (t_{0}-r^{2},t_{0})$ for $(x_{0},t_{0}) \in  \mathbb{R}^{n} \times \mathbb{R} $ and $r>0$. $Q_{r}=Q_{r}(0,0)$ and $Q_{r}^{+}=Q_{r}^{+}(0,0)$. 
\item $Q_{r,\delta}^{+}(x_{0},t_{0}):= B_{r-\delta}^{+}(x_{0}) \times (t_{0}-r^{2}+\delta^{2},t_{0})$, $Q_{r,\delta}^{+}=Q_{r,\delta}^{+}(0,0)$.
\item $V_{r,h}^{+}(x_{0},t_{0}):= B_{r,h}^{+}(x_{0}) \times (t_{0}-r^{2},t_{0})$,  $V_{r,h}^{+}=V_{r,h}^{+}(0,0)$.
\item $V_{r,h,\delta}^{+}(x_{0},t_{0}):= B_{r-\delta,h-\delta h/r}^{+}(x_{0}) \times (t_{0}-r^{2}+\delta^{2},t_{0})$, $V_{r,h,\delta}^{+}=V_{r,h,\delta}^{+}(0,0)$.
\item $Q_{r}^{\ast}(x_{0},t_{0}) = T_{r}(x_{0}) \times (t_{0}-r^{2},t_{0}) $, $Q_{r}^{\ast} := Q_{r}^{\ast} (0,0) $.
\item $K_{r}^{d}=(-r/2, r/2)^{d}$ for $r>0$ and $d= n-1$ or $n$, $K_{r}^{d}(x_{0})=K_{r}^{d} + x_{0}$. 
\item For $|\nu| \le r$, $Q_{1}^{\nu}= Q_{r}(0, 0) \cap (\{ x_{n} > - \nu \} \times \mathbb{R}) $ and $ Q_{r}^{\nu}(x_{0},t_{0})=Q_{r}^{\nu}+(x_{0},t_{0})$.
\item $Q_{r,\delta}^{\nu}= Q_{r, \delta}(0',\nu, 0) \cap (\mathbb{R}_{+}^{n} \times \mathbb{R})  $,   $ Q_{r, \delta}^{\nu}(x_{0},t_{0})=Q_{r, \delta}^{\nu}+(x_{0},t_{0})$.
\item For $\Omega \subset \mathbb{R}^{n}\times \mathbb{R}$, $\partial_{p} \Omega$ is the parabolic boundary of $\Omega$.
\item For $\Omega \subset \mathbb{R}^{n}\times \mathbb{R}$, $  r\Omega := \{ (rx, r^{2}t) \in \mathbb{R}^{n} \times \mathbb{R} : (x,t) \in \Omega \}$ and
$  r\Omega  (x,t):= r\Omega + (x,t)$.
\item We denote the time derivative, gradient and Hessian of $u$ by $u_{t}$, $ Du=(D_{1}u, \cdots , D_{n}u)$, and $D^{2}u = (D_{ij}u) $, respectively, where $ D_{i}u=\frac{\partial u}{\partial x_{i}} $ and $ D_{ij}u=\frac{\partial^{2} u}{\partial x_{i}\partial x_{j}} $ for $ 1 \le i,j \le n$.
\item Let $A$ be a measurable set in $ \mathbb{R}^{n} \times \mathbb{R}$ with $|A| \neq 0 $, and $f$ be a measurable function on $A$. Then we write
$$ \kint_{A} f dxdt = \frac{1}{|A|} \int_{A} f dxdt. $$
\item Let $\Omega \subset \mathbb{R}^{n}\times \mathbb{R}$. If a function $u$ is continuous in $ \Omega $, we write $u \in C(\Omega)$. The $C$-norm of $u$ is given by $$||u||_{C(\Omega)} := \sup_{(x,t) \in \Omega}|u(x,t)|.$$
If $Du $ ($D^{2}u$ and $u_{t}$) is continuous in $\Omega$, we write  $ C^{1}(\Omega)$ ($ C^{2}(\Omega)$, respectively). We define the $C^{1}$,$C^{2}$-norm by
 $$||u||_{C^{1}(\Omega)} := ||u||_{C(\Omega)} +||Du||_{C(\Omega)} ,$$
  $$||u||_{C^{2}(\Omega)}: = ||u||_{C^{1}(\Omega)}+||u_{t}||_{C(\Omega)}+ ||D^{2}u||_{C(\Omega)} .$$
\item   If a function $u$ satisfies $$|u(x,t)-u(y,s)| \le C(|x-y|^{\alpha}+|t-s|^{\alpha/2})$$ for any $(x,t),(y,s) \in \Omega$ and some $0<\alpha \le 1$ and $C>0$, we write $u \in  C^{0,\alpha}(\Omega)$. (i.e, $u$ is $(\alpha/2)$-H\"{o}lder continuous in $t$ and $\alpha$-H\"{o}lder continuous in $x$) The $ C^{0,\alpha}$-norm of $u$ is given by 
\begin{align*}||u||_{ C^{0,\alpha}( \Omega )}& := ||u||_{C( \Omega )}+ \sup_{\substack{(x,t) ,(y,s) \in  \Omega \\ (x,t) \neq (y,s)}}\frac{|u(x,t)-u(y,s)|}{ |x-y|^{\alpha}+|t-s|^{\alpha/2} }\\ &  = : ||u||_{C( \Omega )} + [u]_{C^{0, \alpha}( \Omega )}.
\end{align*}
\item If a function $u$ is $((1+\alpha)/2)$-H\"{o}lder continuous in $t$ and $Du$ is $\alpha$-H\"{o}lder continuous in $x$, we write $u \in C^{1,\alpha}( \Omega )$.  
\begin{align*}||u||_{C^{1,\alpha}( \Omega )} &:= ||u||_{C^{1}( \Omega )}  + \sup_{\substack{(x,t) ,(x,s) \in  \Omega \\ t \neq s}}\frac{|u(x,t)-u(x,s)|}{|t-s|^{(1+\alpha)/2}} \\ & \qquad  \quad + \sum_{i=1}^{n}\sup_{\substack{(x,t) ,(y,s) \in  \Omega \\ (x,t) \neq (y,s)}}\frac{|D_{i}u(x,t)-D_{i}u(y,s)|}{|x-y|^{\alpha}+|t-s|^{\alpha/2}} \\ & = :  ||u||_{C^{1}( \Omega )} + [u]_{C^{1+\alpha}(\Omega)}.
\end{align*}
\item  If a function $u$ satisfies that $u_{t}$ is $(\alpha/2)$-H\"{o}lder continuous in $t$ and $D^{2}u$ is $\alpha$-H\"{o}lder continuous in $x$, we write $u \in C^{2,\alpha}( \Omega )$.  The $C^{2,\alpha}$-norm of $u$ is given by
\begin{align*}||u||_{C^{2,\alpha}( \Omega )} &:= ||u||_{C^{2}( \Omega )}  + \sup_{\substack{(x,t) ,(x,s) \in  \Omega \\ t \neq s}}\frac{|u_{t}(x,t)-u_{t}(x,s)|}{ |x-y|^{\alpha}+|t-s|^{\alpha/2}} \\ & \qquad  \quad + \sum_{i,j=1}^{n}\sup_{\substack{(x,t) ,(x,s) \in  \Omega \\ (x,t) \neq (y,s)}}\frac{|D_{ij}u(x,t)-D_{ij}u(y,s)|}{|x-y|^{\alpha}+|t-s|^{\alpha/2}} \\ & = :  ||u||_{C^{2}( \Omega )} + [u]_{C^{2+\alpha}(\Omega)}.
\end{align*}
\item Let $1\le p \le \infty$. If a function $u$ satisfies that $\int_{\Omega}|u(x,t)|^{p}dxdt < \infty$, we write $u \in L^{p}(\Omega) $. The $L^{p}$-norm of $u$ is given by
\begin{align*}
||u||_{L^{ p}( \Omega )}:=\bigg( \int_{\Omega}|u(x,t)|^{p} dxdt \bigg)^{1/p}.
\end{align*}

In addition, if a function $u$ satisfies that $\esssup_{(x,t) \in \Omega}  |u(x,t)|< \infty $, we write $u \in L^{\infty}(\Omega)$ with its norm
$$||u||_{L^{ \infty}( \Omega )}:= \esssup_{(x,t) \in \Omega}|u(x,t)|.$$
\item If a function $u$ satisfies that $u,  Du \in L^{p}(\Omega)$, we write $u \in W^{1,p}(\Omega) $. The $W^{1,p}$-norm of $u$ is given by
\begin{align*}
||u||_{W^{1,p}( \Omega )}:=\big(||u||_{L^{p}( \Omega )}^{p}+ ||Du||_{L^{p}( \Omega )}^{p}  \big)^{1/p}.
\end{align*}
\item If a function $u$ satisfies that $u, u_{t}, Du, D^{2}u \in L^{p}(\Omega)$, we write $u \in W^{2,p}(\Omega) $. The $W^{2,p}$-norm of $u$ is given by
\begin{align*}
||u||_{W^{2,p}( \Omega )}:=\big(||u||_{L^{p}( \Omega )}^{p}+||u_{t}||_{L^{p}( \Omega )}^{p}+||Du||_{L^{p}( \Omega )}^{p}+||D^{2}u||_{L^{p}( \Omega )}^{p} \big)^{1/p}.
\end{align*}
\item $S(n)$ is denoted by the space of $n \times n $ real symmetric matrices.
\item For every $ M \in S(n)$, $||M||:=\sup_{|x| \le 1} |Mx|$.
\end{itemize}

\subsection{Background knowledge}
We first look at related concepts to proceed with our discussion.
\subsubsection{Parabolic second order differentiability}
To observe second order differentiability, we first need to characterize paraboloids.
\begin{definition}
Let $M >0$.
A \emph{convex paraboloid $P$ with opening $M$} is defined by 
$$P(x,t) = a + l\cdot x + \frac{M}{2} (|x|^{2}-t) , $$
where $a \in \mathbb{R} $ and $l \in \mathbb{R}^{n}$.
We also define a \emph{concave paraboloid} by replacing $M$ with $-M$ in the above definition.
\end{definition}

Let $\Omega$ be a bounded domain in $\mathbb{R}^{n} \times \mathbb{R} $, 
$U \subset \Omega$ be an open subset of $\Omega$, $M>0$, and $u \in C(\Omega)$.
For $s \in \mathbb{R}$, we use the following notation $U_{s}=\{ (x,t) \in U : t \le s \} $ temporarily.
Next we define `good set' and `bad set'.
Let $\underline{G}_{M}(u, U)$ be the set of points  $(x_{0},t_{0}) \in U$  which satisfy that there is a concave paraboloid $P$ with opening $M$ 
such that $P(x_{0},t_{0}) =u(x_{0},t_{0}) $ and $ P(x,t) \le u(x,t)$ for any $(x,t) \in U_{t_{0}}$,
and $\underline{A}_{M}(u, U) = U \backslash \underline{G}_{M}(u, U) $.
Analogously, we can define  $\overline{G}_{M}(u, U)$ and  $\overline{A}_{M}(u, U)$  by using a convex paraboloid as a barrier. In addition, we denote by
$$ G_{M}(u,U) =  \underline{G}_{M}(u, U)  \cap \overline{G}_{M}(u, U),$$
$$ A_{M}(u,U) =  \underline{A}_{M}(u, U)  \cup \overline{A}_{M}(u, U).$$

Roughly speaking, $A_{M}$ can be understood to be a set of points with `bad Hessian'. 
Thus, we need to obtain uniform estimates for its measure to establish $W^{2,p}$-theory,
which will be our main purpose investigated in Section 3 and 4.

\subsubsection{Viscosity solutions}
Let $\Omega \subset \mathbb{R}^{n}$ be a bounded domain and  $\Gamma \subset \partial \Omega$. 
Consider the following problem
\begin{align} \label{paobpl_add}
\left\{ \begin{array}{ll}
F(D^{2}u, Du, u,  x, t) - u_{t} = f & \textrm{in $\Omega_{1}=\Omega \times (0,1)$,}\\
\beta \cdot Du = 0 & \textrm{on $\Gamma_{I} =\Gamma \times I,$}
\end{array} \right. 
\end{align}
where $I$ is a fixed interval in $(0,1)$.

Then a viscosity solution is defined as follows. 
\begin{definition} Let $F$ be continuous in all variables and $f \in C(\Omega_{1} \cup \Gamma_{I})$. 
A continuous function $u \in  C(\Omega_{1} \cup \Gamma_{I})$ is called a viscosity solution of \eqref{paobpl_add} if the following conditions hold:
\begin{itemize}
\item[(a)] for all $ \varphi \in C^{2}(\Omega_{1} \cup \Gamma_{I}) $ touching $u$ by above at $(x_{0},t_{0}) \in \Omega_{1} \cup \Gamma_{I}$, $$F( D^{2} \varphi (x_{0},t_{0}), D \varphi (x_{0},t_{0}),  u(x_{0},t_{0}),  x_{0},t_{0})-\varphi_{t}(x_{0},t_{0}) \ge f(x_{0},t_{0}) $$ when $(x_{0},t_{0})  \in \Omega_{1}$ and $ \beta(x_{0},t_{0})\cdot D\varphi(x_{0},t_{0}) \ge 0$ when $(x_{0},t_{0}) \in \Gamma_{I} $. 
\\ \item[(b)] for all $ \varphi \in C^{2}(\Omega_{1} \cup \Gamma_{I}) $ touching $u$ by below at $(x_{0},,t_{0}) \in \Omega_{1} \cup \Gamma_{I}$, $$F( D^{2} \varphi (x_{0},t_{0}), D \varphi (x_{0},t_{0}),  u(x_{0},t_{0}),  x_{0},t_{0})-\varphi_{t}(x_{0},t_{0}) \le f(x_{0},t_{0}) $$ when $(x_{0},t_{0}) \in \Omega_{1}$ and $ \beta(x_{0},t_{0})\cdot D\varphi(x_{0},t_{0}) \le 0$ when $(x_{0},t_{0}) \in \Gamma_{I} $. 
\end{itemize}
\end{definition}
Note that if a function $u$ satisfies the condition (a) ((b), respectively) in the above definition, we mean that 
$F(D^{2}u, Du, u,  x, t) - u_{t} \ge (\le) f $ in the viscosity sense.
We also remark that viscosity solutions can be defined by using test functions $\varphi$ in $W^{2,p}$ space.
(See \cite{MR1135923})

To talk about viscosity solutions, it is indispensable to introduce the following.
\begin{definition} 
For any $M \in S(n) $, the Pucci extremal operator $ \mathcal{M}^{+} $ and  $\mathcal{M}^{-}$ are defined as follows:
$$  \mathcal{M}^{+}(\lambda, \Lambda, M)=\Lambda \sum_{e_{i}>0} e_{i} + \lambda \sum_{e_{i}<0} e_{i} \ \textrm{and}\  \mathcal{M}^{-}(\lambda, \Lambda, M)=\lambda \sum_{e_{i}>0} e_{i} + \Lambda \sum_{e_{i}<0} e_{i} $$
where $ e_{i}$ are eigenvalues of $M$.
\end{definition}
For $b\ge0$ and $u  $ be a continuous function in the viscosity sense, we also define 
$$ L ^{\pm}(\lambda, \Lambda, b, u ) = \mathcal{M}^{\pm}(\lambda, \Lambda, D^{2}u) \pm b|Du| -u_{t}.$$

Next, we present an important concept to understand viscosity solutions.
\begin{definition}
Let $\Omega \subset \mathbb{R}^{n} \times \mathbb{R}$ , $ b \ge 0$ and $0 < \lambda \le \Lambda $.
We define the classes $\underline{S} (\lambda, \Lambda, b, f) $ ($ \overline{S} (\lambda, \Lambda, b, f) $, respectively) to be the set of all continuous functions $u$ that satisfy $ L ^{+}u \ge f$($ L ^{-}u \le f$) in the viscosity sense in $\Omega$. 
We also define
$$S (\lambda, \Lambda, b, f) =\overline{S} (\lambda, \Lambda, b, f) \cap \underline{S}  (\lambda, \Lambda, b, f)  $$ 
and
$$S^{\ast} (\lambda, \Lambda, b, f) =\overline{S} (\lambda, \Lambda, b,| f|) \cap \underline{S}  (\lambda, \Lambda, b, -|f|) . $$ 
When $b=0$, we omit it like  $S  (\lambda, \Lambda, f)$.
\end{definition}

In this paper, we always assume the following conditions:
$F(X,p,r,x,t)$ is convex in $M$, continuous in $M,q,r,x $ and $t$,  and satisfies
\begin{align} \label{paob_sc} \begin{split}
& \mathcal{M}^{-} (\lambda, \Lambda, M-N) -b|q_{1}-q_{2}| - c|r_{1}-r_{2}| \\ & 
\qquad \le F(M,q_{1},r_{1},x,t) - F(N,q_{2},r_{2},x,t) \\ &
\qquad \qquad \le \mathcal{M}^{+} (\lambda, \Lambda, M-N) +b|p-q| + c|r-s|
\end{split}
\end{align}
for fixed $0< \lambda \le \Lambda $ and $b,c>0$, $M,N \in S(n) $, and any $ q_{1}, q_{2} \in \mathbb{R}^{n}$, $r_{1}, r_{2} \in \mathbb{R}$ and $ (x,t) \in \mathbb{R}^{n} \times \mathbb{R}$.

\section{Boundary estimates for paraboloids}
This section concerns some geometric and analytic tools which will be used later for our $W^{2,p}$-estimates.
We start with several lemmas and prove them, if necessary.

First of all, we mention about some properties of $L^{p}$ functions. 
We can find this result in several papers, for example, \cite[Theorem 4.8]{MR1135923}.
\begin{proposition} \label{mppara}
Let  $f$ be a locally integrable function in $\mathbb{R}^{n}\times \mathbb{R}$ and $\Omega$ be a bounded domain in $ \mathbb{R}^{n}\times \mathbb{R}$.
The maximal operator $M$ is defined as follows:
$$ M(f)(x,t) = \sup_{\rho > 0}  \kint_{Q_{\rho}(x,t)}|f(x,t)|dxdt.$$
Then
$$ |\{ (x,t) \in \Omega: M(f)(x,t) \ge \lambda \}| \le \frac{C(n)}{\lambda} || f||_{L^{1}(\Omega)}$$
  and
$$||M(f)||_{L^{p}(\Omega) } \le C(n,p)||f||_{L^{p} (\Omega)},$$
whenever $f \in L^{p}(\Omega)$ for $ 1 < p< \infty$.
\end{proposition}

The following lemma provides an equivalent condition under which a function is in $L^{p}$ space,
as used in \cite{MR1135923}.
Since the proof is almost same as in the elliptic case, we omit the proof. (See \cite[Lemma 7.3]{MR1351007})
\begin{proposition} \label{paralpeq}
Let $f$ be a nonnegative and measurable function in a domain $ \Omega \subset \mathbb{R}^{n} \times \mathbb{R}$ and $ \mu_{f}$ be its distribution function, that is,
$$ \mu_{f} (\lambda)= | \{ (x,t) \in \Omega : f(x,t) > \lambda \} | \quad \textrm{for} \  \lambda>0 .$$
Let $ \eta >0$ and $M>1$ be given constants. Then, for $ 0< p < \infty $,
$$  f \in L^{p}(\Omega) \ \Longleftrightarrow  \ \sum_{k \ge 1}M^{pk} \mu_{f}(\eta M^{k}) =:S_{0} < \infty$$
and
$$ C^{-1}S_{0} \le ||f||_{L^{p}(\Omega)}^{p} \le C(|\Omega| +S_{0}),$$
where $C$ is a universal constant. 
\end{proposition}

Next we focus on parabolic Hessian estimates.
In \cite{MR1135923}, we can find the following interior and boundary estimates with slight modifications.

\begin{lemma} \label{pardiin}
Let $ u \in \overline{S} (f)$ in $\Omega \times (0,1]$ for some domain $ \Omega \in \mathbb{R}^{n}$.

If $||u||_{ L^{\infty}( \Omega_{1})} \le 1$, then for any $\Omega '  \subset \subset \Omega$ and $\tau \neq 1 $,  
\begin{align*} |\underline{A}_{s} (u, & \Omega \times (0,1]) \cap (\Omega ' \times (0, \tau ])| \\ & \le C(n, \lambda, \Lambda, \Omega, \tau, d(\Omega ', \partial \Omega) )\frac{(1+ ||f||_{ L^{n+1}( \Omega \times (0,1])})^{\mu}}{ s^{\mu}},
\end{align*}
where $ \mu$ is  universal.
\end{lemma}

\begin{lemma} \label{pardibd}
Let $u \in \overline{S} (f)$ in $\Omega = K_{4}^{n-1} \times (0,2) \times (0,2]$.
Suppose $||u||_{ L^{\infty}( \Omega )} \le 1$.
Then
$$ |\underline{A}_{s}  (u,  \Omega) \cap (K_{2}^{n-1} \times (0,1) \times (0, 1])| \le C(n, \lambda, \Lambda   )\frac{(1+ ||f||_{ L^{n+1}( \Omega)})^{\mu}}{ s^{\mu}}.$$
\end{lemma}

By using scaling argument, we can derive the next lemma as a direct consequence of the above results.
\begin{lemma} \label{parlem1}
Let $ \Omega = B_{12 \sqrt{n}}^{+} \times (0,13] $, $0<r \le 1$, and $(x_{0},t_{0}) \in T_{12\sqrt{n}} \times (0,13] $ such that $ r\Omega(x_{0},t_{0})= B_{12 r \sqrt{n}}^{+} \times (t_{0} , t_{0}+13r^{2}] \subset \Omega$.
Assume that $u \in \overline{S} (f)$ in $  r \Omega(x_{0},t_{0}) $, $u \in C(\Omega ) $ and  $||u||_{ L^{\infty}( \Omega )} \le 1$.

Then there exist universal constants $M>1$ and $ 0< \sigma < 1$ such that if
 $$\bigg(\kint_{  r \Omega(x_{0},t_{0})} |f(x,t)|^{n+1}dxdt\bigg)^{\frac{1}{n+1}} \le 1,$$ then we have
\begin{align} \label{pardiest} \frac{|\underline{G}_{M} (u, \Omega ) \cap ( K_{r}^{n-1} \times (0,r) \times (0,r^{2})+ (x_{1},t_{1}))|}{|K_{r}^{n-1} \times (0,r) \times (0,r^{2})|} \ge  1- \sigma, 
\end{align}
whenever $( x_{1},t_{1}) \in (B_{9\sqrt{n}}(x_{0}) \cap \{ x_{n} \ge 0 \}) \times [t_{0}, t_{0}+10r^{2}]$.
\end{lemma} 


The next lemma shows that if there is a point with opening $1$, then the density of `good sector' is guaranteed large enough.
\begin{lemma} \label{palemh}
Under the same hypotheses as in Lemma \ref{parlem1},
we further assume that $u \in S^{\ast}(f)$ in $ r \Omega(x_{0},t_{0}) $, $u \in C(\Omega ) $, and
$$ G_{1} (u, \Omega ) \cap ( K_{3r}^{n-1} \times (0,3r) \times (r^{2},10r^{2})+ (\tilde{x}_{1} ,\tilde{t}_{1}))  \neq \varnothing$$
for some $(\tilde{x}_{1} ,\tilde{t}_{1}) \in (B_{9r \sqrt{n}}(x_{0}) \cap \{ x_{n} \ge 0 \}) \times  [t_{0}, t_{0}+5r^{2}] $.

Then there exist universal constants $M>1$ and $0< \sigma < 1$ such that if
$$\bigg(\kint_{  r \Omega(x_{0},t_{0})} |f(x,t)|^{n+1}dxdt\bigg)^{\frac{1}{n+1}} \le 1,$$ then we have
 $$ \frac{ |G_{M} (u, \Omega ) \cap ( K_{r}^{n-1} \times (0,r) \times (0,r^{2})+ (x_{1},t_{1}))|}{|K_{r}^{n-1} \times (0,r) \times (0,r^{2})|} \ge 1 - \sigma$$
 for any $ ( x_{1},t_{1}) \in  (B_{9 r\sqrt{n}}(x_{0}) \cap \{ x_{n} \ge 0 \} ) \cap [t_{0}, \tilde{t}_{1}]  $. 
\end{lemma}
\begin{proof}
Let $(x_{2},t_{2}) \in G_{1} (u, \Omega ) \cap ( K_{3r}^{n-1} \times (0,3r) \times (r^{2},10r^{2})+ (\tilde{x}_{1} ,\tilde{t}_{1}))  $.
By the definition of $G_{1}$, we have paraboloids with opening $1$ touching $u$ at $(x_{2},t_{2})$ from above and below.
Then we can find a linear function (on $x$) $L$ such that
$$ |u(x,t)- L(x)| \le \frac{1}{2}(|x-x_{2}|^{2} -(t-t_{2}) ).$$

Define $v(x,t) = (u(x,t)- L(x)) / C$ where $C=C(n )$ is a constant so that 
\begin{align*}
||v||_{L^{\infty}(B^{+}_{12r\sqrt{n}}(x_{0}) \times (t_{0}, t_{2}))} \le 1
\end{align*}
and
\begin{align} \label{esolem}
|v(x,t)|  \le  |x|^{2} + t_{2} - t \qquad \textrm{in} \  (B^{+}_{12\sqrt{n}} \backslash B^{+}_{12r\sqrt{n}}(x_{0})) \times (0, t_{2}) .
\end{align}
Now we can see that $v \in S ^{\ast}(f/C) $ in $ B^{+}_{12r\sqrt{n}} \times [t_{0},t_{2})$.
Then we have
$$ \frac{ |G_{M} (v, B^{+}_{12r\sqrt{n}} \times (t_{0},t_{1})) \cap ( K_{r}^{n-1} \times (0,r) \times (0,r^{2})+ (x_{1},t_{1}))|}{ |K_{r}^{n-1} \times (0,r) \times (0,r^{2})|} \ge 1- \sigma. $$
 by Lemma \ref{parlem1}.

Combining the above estimate with \eqref{esolem}, we observe that
$$ \frac{ |G_{N} (v, \Omega ) \cap ( K_{r}^{n-1} \times (0,r) \times (0,r^{2})+ (x_{1},t_{1}))| }{|K_{r}^{n-1} \times (0,r) \times (0,r^{2})|}\ge 1- \sigma $$
for some $N \ge M$.
We also deduce that
$$ G_{M}(v, \Omega  )=G_{MC(n)}(u, \Omega ),$$
and this completes the proof.
\end{proof}

In the elliptic case, Calder\'{o}n-Zygmund decomposition is very useful in establishing $W^{2,p}$ theory. 
We intend to apply this tool in the parabolic case.
One can find the following definition and parabolic decomposition lemma in \cite{MR3185332}.
\begin{definition} Given $m \in \mathbb{N}$, and a dyadic cube $K$ of $Q$, 
the set $\overline{K}^{m}$ is obtained by stacking $m$ copies of its predecessor  $\overline{K}$.
More precisely, if the predecessor $\overline{K}$ has the form $ L \times (a,b) $, then 
$\overline{K}^{m} = L \times  (b,b+m(b-a)) $.
\end{definition}

\begin{lemma} \label{calzypa}
Let $m \in \mathbb{N}$. 
Consider two subsets $A$ and $B$ of a cube $Q$.
Assume that $|A| \le \delta |Q|$ for some $\delta \in (0,1)$.
Assume also the following: for any dyadic cube $K \subset Q $,
$$ |K \cap A| > \delta |K| \Rightarrow \overline{K}^{m} \subset B.$$
Then $|A| \le \frac{m+1}{m}\delta|B| $.
\end{lemma}

Using the above lemma, we can prove the following result. 
\begin{lemma} \label{calzypaap}
Under the same hypotheses as in Lemma \ref{parlem1},
we further assume that $u \in \overline{S} (f)$ in $ r \Omega(x_{0},t_{0}) $, $u \in C(\Omega ) $, $||u||_{ L^{\infty}( \Omega )} \le 1$ and  $$\bigg(\kint_{  r \Omega(x_{0},t_{0})} |f(x,t)|^{n+1}dxdt\bigg)^{\frac{1}{n+1}} \le 1.$$

Extend $f$ by zero outside $ r \Omega(x_{0},t_{0})  $ and define
\begin{align*} A:= A_{M^{k+1}}(u,\Omega ) \cap (K_{r}^{n-1} \times (0,r) \times (t_{1} ,t_{1}+r^{2}) ),
\end{align*}
\begin{align*} B:= &(A_{M^{k}}(u,\Omega ) \cap (K_{r}^{n-1} \times (0,r) \times (t_{1} ,t_{1}+r^{2}))\\ & \cup \{ (x,t) \in K_{r}^{n-1} \times (0,r) \times (t_{0} ,t_{0}+r^{2}) : M(|f|^{n+1})(x,t) \ge (c_{0}M^{k})^{n+1} \}
\end{align*}
for any $k \in \mathbb{N}_{0}$ and $t_{1} \in [t_{0}, t_{0}+5r^{2}]$.

Then $|A| \le 2 \sigma |B|$, where $c_{0}=c_{0}(n )$, $ 0< \sigma < 1$ and $M>1$ are universal.
\end{lemma} 
\begin{proof}
By the definition of $A $ and $B$, we know that $$A \subset B \subset  K_{r}^{n-1} \times (0,r) \times (t_{1} ,t_{1}+r^{2}). $$
We also have $|A| \le \sigma |K_{r}^{n-1} \times (0,r) \times (t_{1} ,t_{1}+r^{2})| $ from Lemma \ref{parlem1}. 
Now we prove that for any dyadic cube $K \subset Q$,
$$|K \cap A |> \sigma |K| \quad \textrm{implies} \quad  \overline{K}^{m} \subset B  $$
for some $m \in \mathbb{N}$.

Let $$K=(K_{r/2^{i}}^{n-1} \times (0,r/2^{i}) \times (0,r^{2}/2^{2i} )+(x_{2},t_{2})$$
be a dyadic cube with its predecessor $$\tilde{K}=(K_{r/2^{i-1}}^{n-1} \times (0,r/2^{i-1}) \times (0,r^{2}/2^{2(i-1)} )+(\tilde{x}_{2},\tilde{t}_{2})$$
for some $i \ge 1$.
Suppose that $K$ satisfies $|K \cap A |> \sigma |K| $ but $\overline{K}^{m} \not\subset B$ for any $m$.
Then there is a point $(x_{3},t_{3}) \in \overline{K}^{1} \backslash B$,
that is,
$$(x_{3},t_{3}) \in  \overline{K}^{1} \cap G_{M^{k}}(u,\Omega ) \quad \textrm{and} \quad M(|f|^{n+1})(x_{3},t_{3}) < (c_{0}M^{k})^{n+1}.$$
 
Now we define a transformation $T$ by $$T (y,s)=(\tilde{x}_{2}+2^{-i}y,\tilde{t}_{2}+2^{-2i}s),$$ and set
$ \tilde{u}(y,s) = 2^{2i}M^{-k}u(T(y,s))$ and $ \tilde{f}(y,s) = M^{-k}f(T(y,s))$.
Since $\overline{K}^{1} \subset  K_{r}^{n-1} \times (0,r) \times (t_{1}, t_{1}+r^{2}) $,
we can observe that
$$(r/2^{i})\Omega(\tilde{x}_{2},\tilde{t}_{2}) \subset   r \Omega(x_{1},t_{1}) $$
and $\tilde{u} \in  S ^{\ast}(\tilde{f})$ in $ r \Omega(0,0) $.
We also see that $|\tilde{x}_{2}- x_{3}| < 2^{-(i-1)}r $ and thus 
$B^{+}_{12r\sqrt{n}/2^{i}}(\tilde{x}_{2}) \subset K_{28r\sqrt{n}/2^{i}}^{n}(x_{3}) $.
Now we have the following estimate
$$\bigg(\kint_{  r \Omega(0,0)} |\tilde f(x,t)|^{n+1}dxdt\bigg)^{\frac{1}{n+1}} \le c_{0}C(n) \le 1$$
by direct calculations.
Note that we chose some sufficiently small $c_{0}$ in order to obtain the last inequality.

On the other hand, since $(x_{3},t_{3}) \in  \overline{K}^{1} \cap G_{M^{k}}(u,\Omega )$, we have
$$ G_{1}(\tilde{u}, T^{-1}\Omega ) \cap (K_{3r}^{n-1} \times (0,3r) \times (4r^{2},13r^{2}) ) \neq \varnothing$$
and then the hypothesis of Lemma \ref{palemh} is satisfied in $\Omega $.
Since $x_{2,n} \ge \tilde{x}_{2,n} $, $t_{2,n} \ge \tilde{t}_{2,n} $ and $ | x_{2} - \tilde{x}_{2}| \le r\sqrt{n}/2^{i} $,
we have $$(2^{i}(x_{2} - \tilde{x}_{2},2^{2i}(t_{2} - \tilde{t}_{2})) \in (B_{9r \sqrt{n}} \cap \{ x_{n} \ge 0 \} ) \times [0,3r^{2}].$$
Thus, we get the following quantity
$$ \frac{ |G_{M}(\tilde{u}, T^{-1}\Omega ) \cap (K_{r}^{n-1} \times (0,r) \times (0, r^{2}) + (2^{i}(x_{2}-\tilde{x}_{2}),2^{2i}(t_{2}-\tilde{t}_{2})) )|}{|K_{r}^{n-1} \times (0,r) \times (0, r^{2})|} $$
is less than  $1- \sigma$.
From this estimate, we have
$$|G_{M^{k+1}}(u,\Omega ) \cap K | \ge (1-\sigma) |K| ,$$
and it contradicts our assumption.
Therefore, we can conclude the proof.
\end{proof}

Finally, we get the estimate for the density of `bad sector'.
\begin{corollary} \label{parobdist}
Under the same hypotheses as in Lemma \ref{parlem1},
we further assume that $u \in  S ^{\ast}(f)$ in $ r\Omega(x_{0},t_{0}) $, $u \in C(\Omega ) $, and  $||u||_{ L^{\infty}( \Omega )} \le 1$.
Then there exist universal constants $C$ and $\mu$ such that if $$\bigg(\kint_{  r \Omega(x_{0},t_{0})} |f(x,t)|^{n+1}dxdt\bigg)^{\frac{1}{n+1}} \le 1,$$ then we have
$$ \frac{ |A_{s}(u,\Omega ) \cap (K_{r}^{n-1} \times (0,r) \times (0,r^{2}) + (x_{1},t_{1}))|}{|K_{r}^{n-1} \times (0,r) \times (0,r^{2})|} \le Cs^{-\mu}$$
for any  $(x_{1},t_{1}) \in (B_{9 r\sqrt{n}}(x_{0}) \cap \{ x_{n} \ge 0 \}) \times   [t_{0}, t_{0}+5r^{2}]$. 
\end{corollary}
\begin{proof} 
Without loss of generality, we can assume that $(x_{0},t_{0})=(0,0)$. 
Let 
$$ \alpha_{k}=\frac{|A_{M^{k}} (u,\Omega ) \cap (K_{r}^{n-1} \times (0,r) \times (t_{1} ,t_{1}+r^{2}))|}{|K_{r}^{n-1} \times (0,r) \times (t_{1} ,t_{1}+r^{2})|} ,$$
$$ \beta_{k} \hspace{-0.2em} =\hspace{-0.2em} \frac{| \{ (x,t) \in K_{r}^{n-1} \times (0,r) \times (t_{0} ,t_{0}+r^{2}) : M(|f|^{n+1})(x,t) \ge (c_{0}M^{k})^{n+1} \}|}{|K_{r}^{n-1} \times (0,r) \times (t_{1} ,t_{1}+r^{2})|}. $$
By Lemma \ref{calzypaap}, we have $ \alpha_{k+1} \le 2\sigma (\alpha_{k}+\beta_{k}) $ for any $k \ge 0$. Then it can be derived directly that
$$ \alpha_{k} \le (2\sigma)^{k} + \sum_{i=0}^{k-1} (2\sigma)^{k-i}\beta_{i}.$$ 

On the other hand, we can also obtain
$$\beta_{i} \le C (c_{0}M^{i})^{-(n+1)}\frac{||f||_{L^{n+1}}^{n+1}}{r^{n+2}} \le CM^{-(n+1)i}$$
by using Proposition \ref{mppara}.
Thus, we can observe that 
$$ \alpha_{k} \le (2\sigma)^{k} + C \sum_{i=0}^{k-1} (2\sigma)^{k-i}M^{-(n+1)i} \le (1+ Ck) \max \{2\sigma, M^{-(n+1)}\}^{k},$$
and the right-hand side is estimated by $C(n)M^{-\mu k}$ for some sufficiently small $\mu$.  
We now finish the proof.
\end{proof}

\section{Boundary $W^{2,p}$-estimates}
In this section, we are concerned with $W^{2,p}$-regularity for the following problem
\begin{align} \label{paob_eq}
\left\{ \begin{array}{ll}
F(D^{2}u, x,t) - u_{t} = f & \textrm{in $Q_{1}^{+}$,}\\
\beta \cdot Du = 0 & \textrm{on $Q_{1}^{\ast}$.}\\
\end{array} \right. 
\end{align}

We first consider the `model problem'
\begin{align} \label{paromoeq}
\left\{ \begin{array}{ll}
F(D^{2}u) - u_{t} = f & \textrm{in $Q_{1}^{+}$,}\\
\beta \cdot Du = 0 & \textrm{on $Q_{1}^{\ast}$.}\\
\end{array} \right. 
\end{align}
 
We are interested in the case when $F$ is slightly perturbed in $x$ and $t$ from \eqref{paromoeq}.
Mentioned earlier in the introduction, it is already known that solutions of \eqref{paromoeq} have interior and boundary $C^{2,\alpha}$-regularity.
Since a solution of the model equation is regular enough to have $C^{1,1}$-regularity, we can expect and wish to prove that the solution of \eqref{paob_eq} enjoys the required $W^{2,p}$-regularity.

First of all, we need to describe a measurement for the perturbation of $F$. We denote an oscillation function for $F$ by
\begin{align} \label{defposc} \psi_{F}((x,t), (y,s)) := \sup_{X \in S(n) \backslash \{0\}} \frac{|F(X,x,t)-F(X,y,s)|}{||X||}.
\end{align}

The following is the main theorem of this section.
\begin{theorem}\label{paob_kthm}
Let $u$ be a viscosity solution of \eqref{paob_eq}
 where $F(X,x,t)$ is uniformly elliptic with $\lambda$ and $\Lambda$, convex in $X$, continuous in $X$, $x$ and $t$ and $F(0,x,t)= 0$, 
$\beta \in C^{2}(\overline{Q}_{1}^{\ast}) $ with $\beta \cdot \mathbf{n} \ge \delta_{0}$ for some $\delta_{0} > 0$, and $f \in L^{p}({Q_{1}^{+}}) \cap C({Q_{1}^{+}})$ for $ n +1< p < \infty $.
Then there exist $\epsilon_{0}$ and $C$ depending on $ n, p,  \lambda, \Lambda, \delta_{0}$ and $ ||\beta||_{C^{2}(\overline{Q}_{1}^{\ast})}$ such that
$$\bigg(  \kint_{B_{r}(x_{0},t_{0}) \cap Q_{1}^{+}}     \psi( (x_{0},t_{0}),(x,t))^{n+1} \ dx dt\bigg)^{\frac{1}{n+1}} \le \epsilon_{0} $$
for any $(x_{0},t_{0}) \in Q_{1}^{+}$ and $r>0$ implies $ u \in W^{2,p}( Q_{\frac{1}{2}}^{+} )$, and we have the estimate
\begin{align} \label{mainest_po} ||u||_{ W^{2,p}( Q_{\frac{1}{2}}^{+} )} \le C (||u||_{ L^{\infty}( Q_{1}^{+} )}+||f||_{ L^{p}( Q_{1}^{+} )}).
\end{align}
\end{theorem}

To prove Theorem \ref{paob_kthm}, we use $C^{1,1}$-regularity results for solutions of \eqref{paromoeq},
which is indeed, there are notable results for this problem proved in  \cite{cm2019oblique} by Chatzigeorgiou and Milakis.
We refer to a series of lemmas, Lemma \ref{abptob} - \ref{parobc2a}, from \cite{cm2019oblique}
to present their modifications as follows. 

\begin{lemma}  \label{abptob}
Let $ f \in C(\overline{Q}_{1}^{+}) $, $g \in C(\overline{Q}_{1}^{\ast}) $, and $u \in C(\overline{Q}_{1}^{+})$ satisfy
\begin{align} \label{parobabpt}
\left\{ \begin{array}{ll}
u \in S ^{\ast}(\lambda, \Lambda,  f) & \textrm{in $Q_{1}^{+}$,}\\
\beta \cdot Du = g & \textrm{on $Q_{1}^{\ast}$.}\\
\end{array} \right. 
\end{align}
Suppose that there exists $\xi \in Q_{1}^{\ast}$ such that $\beta \cdot \xi \ge \delta_{0} $. Then
$$ ||u||_{L^{\infty}(Q_{1}^{+})} \le ||u||_{L^{\infty}(\partial_{p}  Q_{1}^{+}\backslash Q_{1}^{\ast})} + C(||g||_{L^{\infty}( Q_{1}^{\ast})} +  ||f||_{L^{n+1}( Q_{1}^{+})}) $$
where $C$ depends only on $n,\lambda, \Lambda$ and $\delta_{0}$.
\end{lemma}

\begin{lemma} \label{parobhol}
Let $ f \in C(\overline{Q}_{1}^{+}) $, $g \in C(\overline{Q}_{1}^{\ast}) $, and $u \in C(\overline{Q}_{1}^{+})$ satisfy 
\eqref{parobabpt}.
Then  $u \in  C^{0,\alpha}(Q_{\frac{1}{2}}^{+})$ and
$$ ||u||_{ C^{0,\alpha}(\overline{Q}_{\frac{1}{2}}^{+})}   \le C (||u||_{L^{\infty}(Q_{1}^{+})}+||f||_{L^{n+1}(Q_{1}^{+})}+||g||_{L^{\infty}(Q_{1}^{\ast})}) $$
where $0 < \alpha < 1 $ and $C>1$ depend only on $ n, \lambda, \Lambda$ and $\delta_{0}$.
\end{lemma}

\begin{lemma} \label{parobc1a}
Let $u \in C(Q_{1}^{+} \cup Q_{1}^{\ast})$ be a viscosity solution of 
\begin{align} \label{paromoeq2}
\left\{ \begin{array}{ll}
F(D^{2}u) - u_{t} = f & \textrm{in $Q_{1}^{+},$}\\
\beta \cdot Du = g & \textrm{on $Q_{1}^{\ast}$}\\
\end{array} \right. 
\end{align}
and $0 < \alpha < \tilde{\alpha}$,
where $0 < \tilde{\alpha} < 1$ is a constant depending only on $ n, \lambda, \Lambda$ and $\delta_{0}$.
Suppose that $f \in L^{q}(Q_{1}^{+})$ with $q> \frac{(n+1)(n+2)}{2}$, $g, \beta \in C^{0,\alpha}(\overline{Q}_{1}^{\ast})$.
Then $u \in C^{1, \alpha}(\overline{Q}_{1/4}^{+})$ and
$$ ||u||_{C^{1,\alpha}(\overline{Q}_{1/4}^{+})}   \le C (||u||_{L^{\infty}(Q_{1}^{+})}+||f||_{ L^{q}(\overline{Q}_{1}^{+})}+||g||_{C^{0, \alpha}(\overline{Q}_{\frac{1}{2}}^{\ast})} ) $$
where $C$ depends only on $ n, \lambda, \Lambda, \delta_{0} , \tilde{\alpha}$ and $||\beta||_{C^{0,\alpha}(\overline{Q}_{\frac{1}{2}}^{\ast})}$.
\end{lemma}

\begin{lemma}  \label{parobc2a}
Let $F$ be convex, $u \in C(Q_{1}^{+} \cup Q_{1}^{\ast})$ be a viscosity solution of 
\eqref{paromoeq2}
 and $0 < \alpha < \tilde{\alpha}$,
where $0 < \alpha < 1$ is a constant depending only on $ n, \lambda, \Lambda$ and $\delta_{0}$.
Suppose that $\beta, g \in C^{1, \tilde{\alpha}}(\overline{Q}_{1}^{\ast})$ and $ f \in  C^{0,\tilde{\alpha}}(\overline{Q}_{1}^{+})$.
Then $u \in C^{2, \alpha}(\overline{Q}_{1/4}^{+})$ and
$$ ||u||_{C^{2,\alpha}(\overline{Q}_{1/4}^{+})}   \le C (||u||_{L^{\infty}(Q_{1}^{+})}+||f||_{ C^{0,\tilde{\alpha}}(\overline{Q}_{1}^{+})}+||g||_{C^{1, \tilde{\alpha}}(\overline{Q}_{\frac{1}{2}}^{\ast})} ) $$
where $C$ depends only on $ n, \lambda, \Lambda, \delta_{0} , \tilde{\alpha}$ and $||\beta||_{C^{1,\tilde{\alpha}}(\overline{Q}_{\frac{1}{2}}^{\ast})}$.
\end{lemma}
 
\begin{remark} Lemma \ref{abptob} and \ref{parobhol} still hold if $ S ^{\ast}(\lambda, \Lambda,  f) $ is replaced with $S ^{\ast}(\lambda, \Lambda, b,  f) $, since $b$ only influences the dependency of constant $C$.
\end{remark}

Next we state and prove the following global H\"{o}lder estimate for model problems which will be used later in Lemma \ref{parobper}.
\begin{lemma} \label{gloho}Let $u \in C(\overline{V}_{1,h_{0}}^{+}) $ be a viscosity solution of
\begin{align} \label{modelex}
\left\{ \begin{array}{ll}
F(D^{2}u) - u_{t}=0 & \textrm{in $V_{1,h_{0}}^{+}$,}\\
\beta \cdot Du =g & \textrm{on $Q_{1}^{\ast}$,}\\
u=\varphi & \textrm{on $\partial_{p}V_{1,h_{0}}^{+} \backslash Q_{1}^{\ast}$,}\\
\end{array} \right. 
\end{align}
where $\beta \in C^{2}(\overline{Q}_{1}^{\ast}), g \in  L^{\infty}(Q_{1}^{\ast}), \varphi \in C^{0,\alpha}(\partial_{p}V_{1,h_{0}}^{+} \backslash Q_{1}^{\ast})$ for some $0< \alpha < 1$ and $h_{0}>0$ is sufficiently small with 
\begin{align} \label{parconcap} \beta(x,t) \cdot  \mathbf{n}(y) <0 \ \textrm{for any } (x,t)\in Q_{1}^{\ast} \textrm{ and }
y \in \partial B_{1,h_{0}}^{+} \backslash T_{1}.
\end{align}
Then $u \in C^{0,\frac{\alpha}{2}}(\overline{V}_{1,h_{0}}^{+}) $ and
$$ ||u||_{C^{0,\frac{\alpha}{2}}(\overline{V}_{1,h_{0}}^{+})} \le C(||\varphi||_{C^{0,\alpha}(\partial_{p}V_{1,h_{0}}^{+} \backslash Q_{1}^{\ast})}+ ||g||_{L^{\infty}(Q_{1}^{\ast})}), $$
where $C$ depends only on $n, \lambda, \Lambda$ and $\delta_{0}$.
\end{lemma}

To prove the above lemma, we need the following one which can be shown by using the results of \cite[Theorem 7]{MR1073054} and \cite[Lemma 4.1, Lemma 4.3]{MR2070626} and the arguments in the proof of \cite[Theorem 3.1]{MR3780142}.
For Neumann problems, see \cite[Proposition 11]{cm2019oblique}.
\begin{lemma} \label{cm11}
Let $\Omega \subset \mathbb{R}^{n}$ be bounded, $T>0$, $\beta \in C^{2}(\overline{\Gamma})$ with $\Gamma \subset \Omega \times (0,T)$, and $u,v$ satisfy
\begin{align*} 
\left\{ \begin{array}{ll}
F(D^{2}u) - u_{t}\ge f_{1} & \textrm{in $\Omega_{T}:=\Omega \times (0,T)$,}\\
\beta \cdot Du  \ge g_{1} & \textrm{on $\Gamma$,}\\
\end{array} \right. 
\end{align*}
and
\begin{align*} 
\left\{ \begin{array}{ll}
F(D^{2}v) - v_{t} \le f_{2} & \textrm{in $\Omega_{T}$,}\\
\beta \cdot Dv \le g_{2} & \textrm{on $\Gamma$}\\
\end{array} \right. 
\end{align*}
in the viscosity sense, respectively.
Then
\begin{align*}
\left\{ \begin{array}{ll}
u-v \in \underline{S} (\lambda/n,\Lambda, f_{1}-f_{2}) & \textrm{in $\Omega_{T}$,}\\
\beta \cdot D(u-v)   \ge g_{1}- g_{2} & \textrm{on $\Gamma$.}\\
\end{array} \right. 
\end{align*}
\end{lemma}

\begin{proof}[Proof of Lemma \ref{gloho}]
For each $(x_{1},t_{1}) \in \partial_{p}V_{1,h_{0}}^{+} \backslash Q_{1}^{\ast}$, consider
$$w_{1}(x,t) = \varphi(x_{1},t_{1}) + (||\varphi||_{C^{0,\alpha}(\partial_{p}V_{1,h_{0}}^{+} \backslash Q_{1}^{\ast})}+ ||g||_{L^{\infty}(Q_{1}^{\ast})}) \Psi (x,t)$$ and 
$$w_{2}(x,t) = \varphi(x_{1},t_{1}) - (||\varphi||_{C^{0,\alpha}(\partial_{p}V_{1,h_{0}}^{+} \backslash Q_{1}^{\ast})}+ ||g||_{L^{\infty}(Q_{1}^{\ast})})\Psi (x,t),$$ where
$$\Psi(x,t):= K_{1}( (\mathbf{n}(x_{1},t_{1})\cdot (x-x_{1})+K_{2}(t_{1}-t))^{\frac{\alpha}{2}} ),$$  $K_{1}>0$ depends only on $n, \lambda, \Lambda, \delta_{0}$  and $K_{2}=(\sqrt{1+\lambda/2n}-1)/2$.

Then we can check that 
\begin{align*}
\left\{ \begin{array}{ll}
F(D^{2}\Psi) - \Psi_{t} \le 0  & \textrm{in $V_{1,h_{0}}^{+}$,}\\
\beta \cdot D \Psi  \le 0 & \textrm{on $Q_{1}^{\ast}$.}\\
\end{array} \right. 
\end{align*}
Since $u$ is a viscosity solution of \eqref{modelex}, we can observe that
\begin{align*}
\left\{ \begin{array}{ll}
u-w_{1} \in \underline{S} (\lambda/n,\Lambda, 0) & \textrm{in $V_{1,h_{0}}^{+}$,}\\
\beta \cdot D(u-w_{1})   \ge 0 & \textrm{on $Q_{1}^{\ast}$.}\\
u-w_{1} \le 0 & \textrm{on $\partial_{p}V_{1,h_{0}}^{+} \backslash Q_{1}^{\ast}$}\\
\end{array} \right. 
\end{align*}
and
\begin{align*}
\left\{ \begin{array}{ll}
u-w_{2} \in \overline{S} (\lambda/n,\Lambda, 0) & \textrm{in $V_{1,h_{0}}^{+}$,}\\
\beta \cdot D(u-w_{2})   \le 0 & \textrm{on $Q_{1}^{\ast}$.}\\
u-w_{2} \ge 0 & \textrm{on $\partial_{p}V_{1,h_{0}}^{+} \backslash Q_{1}^{\ast}$}\\
\end{array} \right. 
\end{align*}
from Lemma \ref{cm11}.
Then by ABP maximum principle, we have $ w_{2}\le u \le w_{1}$.
This implies
\begin{align*}
|u&(x,t)- \varphi(x_{1},t_{1})| \le  (||\varphi||_{C^{0,\alpha}(\partial_{p}V_{1,h_{0}}^{+} \backslash Q_{1}^{\ast})}+ ||g||_{L^{\infty}(Q_{1}^{\ast})}) \Psi (x,t) 
\\ & \le  K_{1} (||\varphi||_{C^{0,\alpha}+ ||g||_{L^{\infty}(Q_{1}^{\ast})})(\partial_{p}V_{1,h_{0}}^{+} \backslash Q_{1}^{\ast})}  (\mathbf{n}(x_{1},t_{1})\cdot (x-x_{1})+K_{2}(t_{1}-t))^{\frac{\alpha}{2}} 
\\ & \le C (||\varphi||_{C^{0,\alpha}(\partial_{p}V_{1,h_{0}}^{+} \backslash Q_{1}^{\ast})} + ||g||_{L^{\infty}(Q_{1}^{\ast})})
(|x-x_{1}|^{\frac{\alpha}{2}}+|t-t_{1}|^{\frac{\alpha}{4}})
\end{align*}
for some constant $C$ depending on $n, \lambda, \Lambda$ and $\delta_{0}$.
This implies boundary H\"{o}lder regularity.
Now we can complete the proof by combining this estimate with Lemma \ref{parobhol}.
\end{proof}

\begin{remark} \label{rembeta}We have assumed that $\beta \in C^{2}(\overline{\Gamma})$ in Lemma \ref{gloho} and \ref{cm11}, as 
\cite[Lemma 4.1, Lemma 4.3]{MR2070626} hold under this assumption.
\end{remark}

We now fix $h_{0}=h_{0}(n,\delta_{0})>0$ given in Lemma \ref{gloho}.
 
Lemma \ref{parobc2a} and Lemma \ref{gloho} enable us to prove a useful approximation lemma below. 
\begin{lemma} \label{parobper}
Let $0< \epsilon < 1 $ and $u$ be a viscosity solution of
\begin{align*} 
\left\{ \begin{array}{ll}
F(D^{2}u, x,t) - u_{t} = f & \textrm{in $Q_{1}^{+}$,}\\
\beta \cdot Du = g & \textrm{on $Q_{1}^{\ast}$.}\\
\end{array} \right. 
\end{align*}
Assume that $||u||_{L^{\infty}(V_{1,h_{0}}^{+})} \le 1$, $g \in C^{2}(Q_{1}^{\ast})$ and $||\psi( (\cdot,\cdot),(0,0))||_{L^{n+1}(V_{1,h_{0}}^{+})} \le \epsilon$. 
Then, there exists a function $h \in C^{2}(\overline{V}_{\frac{3}{4},\frac{3}{4}h_{0}}^{+})$ 
such that
$ u- h \in S(\varphi) $,
$ ||h||_{C^{2}(\overline{V}_{\frac{3}{4},\frac{3}{4}h_{0}}^{+})} \le C $ and
$$ ||u-h||_{L^{\infty}(V_{\frac{3}{4},\frac{3}{4}h_{0}}^{+})} + || \varphi||_{L^{n+1}(V_{\frac{3}{4},\frac{3}{4}h_{0}}^{+})}  \le C (\epsilon ^{\gamma}+||f||_{L^{n+1}(V_{1,h_{0}}^{+})}) $$
 for some $0 < \gamma = \gamma(n, \lambda, \Lambda, \delta_{0}) <1$ and $$C = C(n, \lambda, \Lambda, \delta_{0}, ||\beta||_{C^{2}(\overline{Q}_{1}^{\ast})}, ||g||_{C^{2}(Q_{1}^{\ast})}).$$
Here, $\varphi(x,t)=f(x,t)- F(D^{2}h(x,t), x,t)+F(D^{2}h(x,t),0,0)$.
\end{lemma}
\begin{proof} 
Let $h$ be a solution of
\begin{align}  
\left\{ \begin{array}{ll}
F(D^{2}h, 0,0) -h_{t}= 0 & \textrm{in $V_{\frac{7}{8},\frac{7}{8}h_{0}}^{+}$,}\\
h=u & \textrm{on $\partial_{p} V_{\frac{7}{8},\frac{7}{8}h_{0}}^{+} \backslash Q_{\frac{7}{8}}^{\ast} $,} \\
\beta \cdot Dh = g & \textrm{on $Q_{\frac{7}{8}}^{\ast}$.}\\
\end{array} \right. 
\end{align}
Applying Lemma \ref{parobhol} to $u$, we can obtain
\begin{align} \label{parholest}
 ||u||_{C^{0,\alpha_{1}}(V_{\frac{7}{8},\frac{7}{8}h_{0}}^{+})}   \le C (1+||f||_{L^{n+1}(V_{1,h_{0}}^{+})})
\end{align}
for some $\alpha_{1}=\alpha_{1}(n, \lambda, \Lambda, \delta_{0})$ and $C = C(n, \lambda, \Lambda, \delta_{0}, ||g||_{C^{2}(Q_{1}^{\ast})}) $.
We also have  
\begin{align*}
 ||h||_{L^{\infty}(V_{\frac{7}{8},\frac{7}{8}h_{0},\delta}^{+})} &+ \delta||Dh||_{L^{\infty}(V_{\frac{7}{8},\frac{7}{8}h_{0},\delta}^{+})} \\ & + \delta^{2}(||h_{t}||_{L^{\infty}(V_{\frac{7}{8},\frac{7}{8}h_{0},\delta}^{+})} + ||D^{2}h||_{L^{\infty}(V_{\frac{7}{8},\frac{7}{8}h_{0},\delta}^{+})} ) \le C ,
\end{align*}
where $C$ is a constant depending on $n , \lambda , \Lambda, \delta_{0}, ||\beta||_{C^{2}(\overline{Q}_{1}^{\ast}))}  $ and $  ||g||_{C^{2}(Q_{1}^{\ast})}$ 
by means of Lemma \ref{parobc2a} with scaling.

We set $ w = u -h$. Then, $w$ satisfies
\begin{align}  
\left\{ \begin{array}{ll}
w \in S(\lambda/n, \Lambda, \varphi) & \textrm{in $V_{\frac{7}{8},\frac{7}{8}h_{0}}^{+}$,}\\
w=0 & \textrm{on $\partial_{p} V_{\frac{7}{8},\frac{7}{8}h_{0}}^{+} \backslash Q_{\frac{7}{8}}^{\ast} $,} \\
\beta \cdot Dw = 0 & \textrm{on $Q_{\frac{7}{8}}^{\ast}$.}\\
\end{array} \right. 
\end{align}
Apply Lemma \ref{abptob} to $w$, we get
\begin{align*}
 ||w&||_{L^{\infty}(V_{\frac{7}{8},\frac{7}{8}h_{0}, \delta }^{+})}  \le C ( ||\varphi||_{L^{n+1}(V_{\frac{7}{8},\frac{7}{8}h_{0},\delta}^{+})}+ ||w||_{L^{\infty}(\partial_{p} V_{\frac{7}{8},\frac{7}{8}h_{0}, \delta }^{+} \backslash Q_{\frac{7}{8}}^{\ast}  ) }) \\
& \le  C ( ||f||_{L^{n+1}(V_{\frac{7}{8},\frac{7}{8}h_{0}, \delta }^{+})}+ ||F(D^{2}h, \cdot, \cdot )-F(D^{2},0,0)||_{L^{n+1}(V_{\frac{7}{8},\frac{7}{8}h_{0}, \delta }^{+})}
 \\ & \qquad \qquad + ||w||_{L^{\infty}(\partial_{p} V_{\frac{7}{8},\frac{7}{8}h_{0}, \delta }^{+}\backslash Q_{\frac{7}{8}}^{\ast} ) })
\end{align*}
for some $C=C(n, \lambda , \Lambda, \delta_{0} ) $.
Observe that
\begin{align*}
||F(D^{2}h, \cdot, \cdot)-&F(D^{2}h,0,0)||_{L^{n+1}(V_{\frac{7}{8},\frac{7}{8}h_{0}, \delta }^{+})} \\ & \le ||\psi((\cdot,\cdot),(0, 0))||_{L^{n+1}(V_{\frac{7}{8},\frac{7}{8}h_{0}, \delta }^{+})} ||D^{2}h||_{L^{\infty}(V_{\frac{7}{8},\frac{7}{8}h_{0}, \delta }^{+})} \\ &
\le C \delta^{-2} \epsilon ,
\end{align*}
where $C=C(n , \lambda , \Lambda, \delta_{0},||\beta||_{C^{2}(\overline{Q}_{1}^{\ast})}, ||g||_{C^{2}(Q_{1}^{\ast})})$.
We have used $h \in C^{2}(\overline{V}_{\frac{7}{8},\frac{7}{8}h_{0}}^{+})$ in the first inequality.

Meanwhile, we can see that  $w \equiv 0$ on $\partial_{p} V_{\frac{7}{8},\frac{7}{8}h_{0}}^{+}\backslash Q_{\frac{7}{8}}^{\ast}  $ and $u \in C^{0,\alpha_{1}}(V_{\frac{7}{8},\frac{7}{8}h_{0}}^{+})$.
Then we also obtain a global H\"{o}lder regularity for $h$ by combining Lemma \ref{gloho} with  \eqref{parholest}.
Now we get
\begin{align*}
||w||_{L^{\infty}(\partial_{p} V_{\frac{7}{8},\frac{7}{8}h_{0}, \delta }^{+} \backslash  Q_{\frac{7}{8}}^{\ast}  ) }
\le   C\delta^{\alpha_{2}}(1+||f||_{L^{n+1}(V_{1,h_{0}}^{+})}) 
\end{align*}
for some  $ \alpha_{2} \in (0,\alpha_{1})$ and $ C = C(n, \lambda , \Lambda, \delta_{0}, ||g||_{C^{2}(Q_{1}^{\ast})} ) $.
Thus, if we put $\gamma = \delta^{\frac{\alpha_{2}}{2+\alpha_{2}}}$,
\begin{align*}
||w||_{L^{\infty}(\partial_{p} V_{\frac{7}{8},\frac{7}{8}h_{0}, \delta}^{+} ) } &
\le C\{ ||f||_{L^{n+1}(V_{1,h_{0}}^{+})}+ \delta^{-2} \epsilon + \delta^{\alpha_{2}} (1+||f||_{L^{n+1}(V_{1,h_{0}}^{+})})\}
\\ & \le C (\epsilon ^{\gamma}+||f||_{L^{n+1}(V_{1,h_{0}}^{+})})
\end{align*}
for some constant $C$ depending only on  $n , \lambda , \Lambda, \delta_{0},||\beta||_{C^{2}( Q_{1}^{\ast} )}$ and $||g||_{C^{2}(Q_{1}^{\ast})}$.
This completes the proof.
\end{proof}

The following lemmas give us
 useful information about solutions of \eqref{paob_eq} in the viscosity sense.
 
\begin{lemma} \label{pl02}
Let $0< \epsilon_{0} < 1 $, $ \Omega = B_{14 \sqrt{n}h_{1}^{-1},14 \sqrt{n} }^{+} \times (0,15] $, $r \le 1$, 
and $u$ be a viscosity solution of 
\begin{align} \label{parobmoeq}
\left\{ \begin{array}{ll}
F(D^{2}u,x, t) - u_{t} = f & \textrm{in $\Omega$,}\\
\beta \cdot Du = 0 & \textrm{on $S:=T_{14 \sqrt{n}h_{1}^{-1}} \times (0 , 15]$,}\\
\end{array} \right. 
\end{align}
where $h_{1}=h_{1}(n, \delta_{0})$ is a small constant satisfying \eqref{parconcap} for any $ (x,t)\in S$ and $ y \in \partial B_{14\sqrt{n}h_{1}^{-1},14\sqrt{n} }^{+} \backslash T_{14\sqrt{n}h_{1}^{-1}}$.
Consider a point $(x_{0},t_{0}) \in S$ with $ r\Omega(x_{0},t_{0})  \subset \Omega$.
Assume that $$\bigg(\kint_{  r \Omega(x_{0},t_{0})} \hspace{-0.5em}|f(x,t)|^{n+1}dxdt\bigg)^{\frac{1}{n+1}} \hspace{-0.15em}+\hspace{-0.15em} \bigg(\kint_{  r \Omega(x_{0},t_{0})}\hspace{-1em} |\psi((x,t),(x_{0},t_{0}))|^{n+1}dxdt\bigg)^{\frac{1}{n+1}}  \hspace{-0.15em} \le \hspace{-0.15em} \epsilon  $$ for some $ \epsilon<1$ depending on $ n, \epsilon_{0}, \lambda, \Lambda, \delta_{0}$ and $ ||\beta||_{C^{2}(Q_{14 \sqrt{n}h_{1}^{-1}}^{\ast})}$.
Then, 
\begin{align} \label{pl02co1} G_{1}(u, \Omega ) \cap (K_{3r}^{n-1} \times (0,3r) \times (r^{2},10r^{2})+(\tilde{x}_{1},\tilde{t}_{1})) \neq \varnothing \end{align}
for some $(\tilde{x}_{1} ,\tilde{t}_{1}) \in (B_{9r \sqrt{n}h_{1}^{-1},9r \sqrt{n}}^{+}(x_{0}) \cup T_{9r \sqrt{n}}(x_{0})) \times  [t_{0}+2r^{2}, t_{0}+5r^{2}] $ implies
$$ \frac{|G_{M}(u, \Omega) \cap (K_{r}^{n-1} \times (0,r) \times (0,r^{2})+(x_{1},t_{1})) }{| K_{r}^{n-1} \times (0,r) \times (0,r^{2})|}   \ge 1- \epsilon_{0}, $$
where $ ( x_{1},t_{1}) \in  (B_{9r \sqrt{n}h_{1}^{-1},9r \sqrt{n}}^{+}(x_{0}) \cup T_{9r \sqrt{n}}(x_{0}) ) \times [t_{0}+2r^{2}, \tilde{t}_{1}]  $
and $M$ is a constant depending only on $n, \lambda, \Lambda, \delta_{0}$ and $ ||\beta||_{C^{2}(Q_{14 \sqrt{n}h_{1}^{-1}}^{\ast})}$.
\end{lemma}

\begin{proof}
From \eqref{pl02co1}, there exists a point $(x_{2},t_{2}) $ such that 
$$(x_{2},t_{2}) \in G_{1}(u, \Omega) \cap  (K_{3r}^{n-1} \times (0,3r) \times (r^{2},10r^{2})+(\tilde{x}_{1},\tilde{t}_{1})).$$
By the definition of $G_{1}$, we can find a linear function $L$ such that
$$ |u(x,t)-L(x)| \le \frac{1}{2}(|x-x_{2}|^{2}-(t-t_{2})) $$
for any $(x,t) \in  B_{14 \sqrt{n}h_{1}^{-1},14 \sqrt{n}}^{+} \times (0,t_{2} ) $.
Let $ \tilde{u}(x,t) = (u(x,t)-L(x) ) / C(n)$ with $\tilde{u}$ satisfying
$|| \tilde{u}||_{L^{\infty}(B_{14r \sqrt{n}h_{1}^{-1},14r \sqrt{n}}^{+}(x_{0}) \times (t_{0}, t_{2} ) )} \le 1$ and 
$$|\tilde{u}(x,t) | \le |x|^{2} -(t-t_{2})  \ \textrm{ in } \ (B_{14 \sqrt{n}h_{1}^{-1},14 \sqrt{n}}^{+} \backslash B_{14r \sqrt{n}h_{1}^{-1},14r \sqrt{n}}^{+}(x_{0}))\times [0, t_{2}].$$ 
Here we can check that
$$||L||_{C^{1}(B_{14r \sqrt{n}h_{1}^{-1},14r \sqrt{n}}^{+}(x_{0}) \times [0, t_{2} ])}  \le C(n) + ||u||_{L^{\infty}(\Omega)}, $$
and thus  $|DL| $ is uniformly bounded and depending only on $n $ and $||u||_{L^{\infty}(\Omega)}$ in this case.

Next we define $\tilde{F}(D^{2} \tilde{u}, x,t) = F(C D^{2} \tilde{u},x,t) /C(n),\ \tilde{f}(x,t) = f(x,t)/C(n)$. 
We see that the elliptic constants of $F$ and $\tilde{F}$ are the same and $\tilde{u}$ is a viscosity solution of
\begin{align} 
\left\{ \begin{array}{ll}
\tilde{F}(D^{2}\tilde{u}, x,t)- \tilde{u}_{t} = \tilde{f} & \textrm{in $ r \Omega(x_{0},t_{0}) $,}\\
\beta \cdot D\tilde{u} = - \beta \cdot DL / C(n)& \textrm{on $rS(x_{0},t_{0})$.}\\
\end{array} \right. 
\end{align}
Set $\Omega' =  B_{14 \sqrt{n}h_{1}^{-1},14 \sqrt{n}h_{1}}^{+} \times (1,15]$, $\Omega'' =  B_{13 \sqrt{n}h_{1}^{-1},13 \sqrt{n}}^{+} \times (2,15]$ and $S' =  T_{14 \sqrt{n}}^{+} \times (1,15]$.
We also write $\Omega'_{\delta} =  B_{(14 -\delta)\sqrt{n}h_{1}^{-1},(14 -\delta)\sqrt{n}}^{+} \times (1+\delta^{2},15]$.
Consider a function $\tilde{h} \in   C(r\Omega'(x_{0},t_{0}))$ which solves
\begin{align} 
\left\{ \begin{array}{ll}
\tilde{F}(D^{2}\tilde{h}, 0,0) - \tilde{h}_{t} = 0 & \textrm{in $r\Omega'(x_{0},t_{0})$,}\\
\tilde{h}=\tilde{u} & \textrm{on $\partial_{p} (r\Omega'(x_{0},t_{0}) )\backslash rS'(x_{0},t_{0}) $,} \\
\beta \cdot D\tilde{h} = - \beta \cdot DL / C(n)  & \textrm{on $rS'(x_{0},t_{0})$}\\
\end{array} \right. 
\end{align}
in the viscosity sense.
Since $\beta \in C^{2}(rS(x_{0},t_{0}))$ and $DL$ is a constant vector, $\beta \cdot DL \in C^{2}(rS(x_{0},t_{0}))$. 
Then we can derive that 
\begin{align*}
 ||\tilde{u}||_{C(r\Omega'(x_{0},t_{0}))} & + r[\tilde{u}]_{C^{0,\alpha_{1}}(r\Omega'(x_{0},t_{0}))} \\ & \le C ( 1+ r^{\frac{n }{n+1}}||f||_{L^{n+1}(r\Omega(x_{0},t_{0}))}+r||\beta \cdot DL||_{L^{\infty}(r\Omega(x_{0},t_{0}))} ) 
\end{align*} for some $\alpha_{1}=\alpha_{1}(n, \lambda, \Lambda, \delta_{0}) \in (0,1)$ and $C=C(n, \lambda, \Lambda, \delta_{0})>0$
by Lemma \ref{parobhol}.
On the other hand, applying Lemma \ref{abptob} and \ref{parobc2a} to $\tilde{h}$, we also have
\begin{align*}
&||\tilde{h}||_{C(r(\Omega'_{\delta})(x_{0},t_{0}))} +  r \delta||D\tilde{h}||_{C(r(\Omega'_{\delta})(x_{0},t_{0}))} \\ & \qquad\qquad\qquad \qquad\qquad +  (r\delta)^{2} ( ||\tilde{h}_{t}||_{C(r(\Omega'_{\delta})(x_{0},t_{0}))}  + ||D^{2}\tilde{h}||_{C(r(\Omega'_{\delta})(x_{0},t_{0}))} )
\\ & \le C ( ||\tilde{h}||_{L^{\infty}(r\Omega'(x_{0},t_{0}))} +  r|| \beta \cdot DL ||_{C(rS(x_{0},t_{0}))} + r^{2}|| D\beta \otimes DL ||_{C(rS(x_{0},t_{0}))}
 \\ & \qquad \qquad  +  r^{2+\alpha}[ D\beta \otimes DL ]_{C^{0,\alpha}(rS(x_{0},t_{0}))} )
\end{align*}
for any $\alpha \in (0,1)$ and some $C$ depending only on $n, \lambda, \Lambda,\delta_{0}$  and $ ||\beta||_{C^{2}(rS(x_{0},t_{0}))} $.
Next, we observe that 
\begin{align*}
 || \beta \cdot DL &||_{C(rS(x_{0},t_{0}))}+ r|| D\beta \otimes DL ||_{C(rS(x_{0},t_{0}))} +  r^{1+\alpha}[ D\beta \otimes DL ]_{C^{0,\alpha}(rS(x_{0},t_{0}))} )\\ & 
 \le C(n,   ||\beta||_{C^{2}(rS(x_{0},t_{0}))})
\end{align*}
and
\begin{align*}
 &||\tilde{h}||_{L^{\infty}(r(\Omega'_{\delta})(x_{0},t_{0}))} \\ &  \le  ||\tilde{h}||_{L^{\infty}(\partial_{p} (r(\Omega'_{\delta})(x_{0},t_{0}))) \backslash rS'(x_{0},t_{0}))} + C(n, \lambda, \Lambda , \delta_{0}) r || \beta \cdot DL ||_{L^{\infty}(rS(x_{0},t_{0}))} \\
 & \le  ||\tilde{u}||_{L^{\infty}(\partial_{p} (r\Omega'(x_{0},t_{0})) \backslash rS'(x_{0},t_{0}))}  + C(n, \lambda, \Lambda , \delta_{0}, ||\beta||_{C^{2}(rS(x_{0},t_{0}))})r  \\ & \   
 +  C(n, \lambda, \Lambda , \delta_{0}) (1+r^{\frac{n}{n+1}}||f||_{L^{n+1}(r\Omega(x_{0},t_{0}))}+r||DL||_{L^{\infty}(r\Omega(x_{0},t_{0}))} ) \delta^{\alpha_{2}}
 \\ & 
 \le C(n, \lambda, \Lambda , \delta_{0},  ||\beta||_{C^{2}(rS(x_{0},t_{0}))})
\end{align*}
for any $0<\delta < 2$ and some $ \alpha_{2} \in (0,\alpha_{1}) $.  
We have used a similar argument for the second inequality in the proof of Lemma \ref{parobper}.
Hence, we get
\begin{align*}||D^{2}\tilde{h}||_{L^{\infty}(r(\Omega'_{\delta})(x_{0},t_{0})))}&+ ||\tilde{h}_{t}||_{L^{\infty}(r(\Omega'_{\delta})(x_{0},t_{0})))} \\ & \le  \delta^{-2} C(n, \lambda, \Lambda, \delta_{0},  ||\beta||_{C^{2}(rS(x_{0},t_{0}))})
\end{align*} for any $0<\delta<2$
and therefore
$$||D^{2}\tilde{h}||_{L^{\infty}( r\Omega''(x_{0},t_{0}))}+ ||\tilde{h}_{t}||_{L^{\infty}( r\Omega''(x_{0},t_{0}))}  \le  C(n, \lambda, \Lambda, \delta_{0},  ||\beta||_{C^{2}(rS(x_{0},t_{0}))}).$$
Then the above estimate leads to
$$ A_{N}(\tilde{h}, r\Omega''(x_{0},t_{0}))\cap (Q_{r}^{n-1} \times (0,r) \times (0,r^{2})+(x_{1},t_{1})) = \varnothing$$
for any $( x_{1},t_{1}) \in  (B_{9r \sqrt{n}} \cap \{ x_{n} \ge 0 \} ) \times [t_{0}+2r^{2}, \tilde{t}_{1}] $ 
and a sufficiently large $N = N(n, \lambda, \Lambda, \delta_{0},  ||\beta||_{C^{2}(rS(x_{0},t_{0}))}) $.

Extend $\tilde{h}|_{r\Omega''(x_{0},t_{0})}$ to $H$ with the property that 
$H$ is continuous in $ \Omega_{t_{2}}$, where 
$$ \Sigma_{s}:= \{ (x,t) \in \Sigma : t \le s \} \quad \textrm{for} \ \Sigma \in \mathbb{R}^{n} \times \mathbb{R},$$
$ H = \tilde{u}$ in $\Omega_{t_{2}}  \backslash (r\Omega'(x_{0},t_{0}))_{t_{2}}$, and
$$ ||\tilde{u} -H||_{L^{\infty}(\Omega_{t_{2}} )}=||\tilde{u} -\tilde{h}||_{L^{\infty}((r\Omega''(x_{0},t_{0}))_{t_{2}})}.$$ 
Then we have
\begin{align*} ||\tilde{u} -H||_{L^{\infty}(\Omega_{t_{2}} )} \le  ||\tilde{u}||_{L^{\infty}((r\Omega''(x_{0},t_{0}))_{t_{2}})}+ ||\tilde{h}||_{L^{\infty}((r\Omega''(x_{0},t_{0}))_{t_{2}})} \le C_{0}
\end{align*}
for some $C_{0} =C_{0}(n, \lambda, \Lambda , \delta_{0}, ||\beta||_{C^{2}(rS(x_{0},t_{0}))}) $.
From this, we see that
$$ |H(x,t)| \le C_{0} + |x|^{2} - (t-t_{2})\qquad \textrm{in} \quad \Omega_{t_{2}} \backslash (r\Omega''(x_{0},t_{0}))_{t_{2}}.$$
It can be obtained directly that 
$$A_{M_{0}}(H, \Omega ) \cap (K_{r}^{n-1} \times (0,r) \times (0,r^{2}) + (x_{1},t_{1})) = \varnothing  $$
 for some $M_{0} \ge N$.

Define $ w= \tilde{u} - H $. Then $w$ satisfies
\begin{align} 
\left\{ \begin{array}{ll}
w \in S (\lambda/n, \Lambda, \tilde{f}-\tilde{F}(D^{2}\tilde{h},\cdot, \cdot)+\tilde{h}_{t}) & \textrm{in $r\Omega'(x_{0},t_{0})$,}\\
w=0 & \textrm{on $\partial_{p} (r\Omega'(x_{0},t_{0}) )_{t_{2}}\backslash (rS'(x_{0},t_{0}))_{t_{2}} $,} \\
\beta \cdot Dw = 0 & \textrm{on $rS'(x_{0},t_{0})$.}\\
\end{array} \right. 
\end{align}
From Lemma \ref{abptob}, we can derive
\begin{align*}  
||w||_{L^{\infty}(\Omega_{t_{2}} ) }  =||w||_{L^{\infty}((r\Omega''(x_{0},t_{0}))_{t_{2})} }  \le C (\epsilon ^{\gamma}+||f||_{L^{n+1}(r\Omega(x_{0},t_{0}))}) \le C \epsilon^{\gamma}
\end{align*}
for some $ \gamma \in (0, 1)$ depending only on $n, \lambda, \Lambda, \delta_{0}$ and $C>1$ which also depends on $ ||\beta||_{C^{2}(rS(x_{0},t_{0}))}$.

Now write $\tilde{w} = w / C\epsilon^{\gamma}$. 
Since $\tilde{w}$ satisfies the assumptions of Corollary \ref{parobdist}, it holds that 
\begin{align*}
\frac{|A_{s}(\tilde{w}, \Omega ) \cap (K_{r}^{n-1} \times (0,r) \times (0,r^{2}) +(x_{1},t_{1})) |}{|K_{r}^{n-1} \times (0,r) \times (0,r^{2})|} \le C s^{- \mu}.
\end{align*}
We also check that $$ A_{2M_{0}}(\tilde{u}, \Omega ) \subset A_{M_{0}}(w, \Omega ) \cup A_{M_{0}}(H, \Omega ) ,$$
$$  A_{M_{0}}(H, \Omega ) \cap (K_{r}^{n-1} \times (0,r) \times (0,r^{2})+(x_{1},t_{1})) = \varnothing. $$
This implies
\begin{align*}
| A_{2M_{0}}&(\tilde{u}, \Omega ) \cap  (K_{r}^{n-1} \times (0,r) \times (0,r^{2})+(x_{1},t_{1})) | \\ & \le
| A_{M_{0}}(w, \Omega ) \cap  (K_{r}^{n-1} \times (0,r) \times (0,r^{2})+(x_{1},t_{1})) |
\\ & = | A_{M_{0}/C\epsilon^{\gamma}}(\tilde{w}, \Omega ) \cap  (K_{r}^{n-1} \times (0,r) \times (0,r^{2})+(x_{1},t_{1})) |
\\ & \le C ( M_{0}/C\epsilon^{\gamma} )^{-\mu} |K_{r}^{n-1} \times (0,r) \times (0,r^{2})|
\\ & \le \epsilon_{0}|K_{r}^{n-1} \times (0,r) \times (0,r^{2})|
\end{align*}
for $M=  2CM_{0} $ and a sufficiently small $\epsilon$. 
Then we get the desired result.
\end{proof}

\begin{lemma} \label{paob_lem3}
Let $0< \epsilon_{0} < 1 $, $ \Omega = B_{14 \sqrt{n}h_{1}^{-1},14 \sqrt{n}}^{+} \times (0,15] $, $r \le 1$, and $u$ be a viscosity solution of 
\eqref{parobmoeq}.
Assume that $ ||u||_{L^{\infty}( r \Omega(x_{0},t_{0}))} \le 1$ and  $$\bigg(\kint_{  r \Omega(x_{0},t_{0})} |f(x,t)|^{n+1}dxdt\bigg)^{\frac{1}{n+1}} \le \epsilon$$ for some
$ \epsilon >0$ depending only on $ n, \epsilon_{0}, \lambda, \Lambda, \delta_{0}, ||\beta||_{C^{2}(rS (x_{0},t_{0}))} $.

Extend $f$ to zero outside $ r \Omega(x_{0},t_{0})$ and let
$$\bigg(  \kint_{Q_{r}(x_{1},t_{1}) \cap r \Omega(x_{0},t_{0})}     \psi( (x_{1},t_{1}),(x,t))^{n+1} \ dxdt \bigg)^{\frac{1}{n+1}} \le \epsilon$$
for any $(x_{1},t_{1}) \in  r \Omega(x_{0},t_{0})$, $r>0$.
Then, for 
\begin{align*} A:= A_{M^{k+1}}(u,\Omega ) \cap (K_{r}^{n-1} \times (0,r) \times (t_{0}+2r^{2} ,t_{0}+3r^{2}) ),
\end{align*}
\begin{align*} B:=& (A_{M^{k}}(u,\Omega ) \cap (K_{r}^{n-1} \times (0,r) \times (t_{0}+2r^{2} ,t_{0}+3r^{2}))\ \cup \\ &  \{ (x,t) \hspace{-0.2em} \in \hspace{-0.2em} K_{r}^{n-1} \hspace{-0.3em} \times \hspace{-0.3em} (0,r) \hspace{-0.1em} \times \hspace{-0.1em} (t_{0}+2r^{2} ,t_{0}+3r^{2}) \hspace{-0.3em}:\hspace{-0.3em} M(|f|^{n+1}) (x,t)\hspace{-0.2em} \ge \hspace{-0.2em} (c_{0}M^{k})^{n+1} \},
\end{align*}
where $k \in \mathbb{N}_{0}$, $M>1$ only depends on $n,  \lambda, \Lambda, \delta_{0},  ||\beta||_{C^{2}(rS (x_{0},t_{0}))}$ and $c_{0}$ also depends on $  \epsilon_{0}$, we have
$$ |A| \le 2\epsilon_{0}|B| .$$
\end{lemma}
\begin{proof}
First of all, we observe that $$ A \subset B \subset K_{r}^{n-1} \times (0,r) \times (t_{0}+2r^{2} ,t_{0}+3r^{2}).$$
We also have $B \subsetneq K_{r}^{n-1} \times (0,r) \times (t_{0}+2r^{2} ,t_{0}+3r^{2})$ by Lemma \ref{palemh}.
Thus, applying Lemma \ref{pl02} to $u$, we obtain $|A| \le 2\epsilon_{0}$.
Then we only need to show that for any parabolic dyadic cube $K$ and its predecessor $\tilde{K}$,
\begin{align*}
|A \cap K | > \epsilon_{0}|K| \quad \Rightarrow \quad \overline{K}^{1} \subset B 
\end{align*}
by means of Lemma \ref{calzypa}.

We define $$K = (K_{r/2^{i}}^{n-1} \times (0, r/2^{i}) \times (0, r^{2}/2^{2i}))+(x_{1},t_{1})$$ and $$\tilde{K} = (K_{r/2^{i-1}}^{n-1} \times (0, r/2^{i-1}))\times (0, r^{2}/2^{2(i-1)}))+(\tilde{x}_{1},\tilde{t}_{1}).$$ 
Suppose that $ |A \cap K | > \epsilon_{0}|K| $ and $ \overline{K}^{1} \nsubseteq B $.
There exists a point $(x_{2},t_{2}) \in \overline{K}^{1} \cap G_{M^{k}}(u,\Omega)$ with 
$ M(|f|^{n+1})(x_{2},t_{2}) < (c_{0}M^{k})^{n+1} $.

First, assume that $x_{1, n} \le 8r\sqrt{n} / 2^{i}$.
Consider a linear transformation $$T(y,s)=(x_{1}',0,t^{\ast})+(2^{-i}y,2^{-2i}s),$$
where $t^{\ast}=t_{1}-2^{1-2i}r^{2}$. 
Now we set $$\tilde{u}(y,s)= 2^{2i}M^{-k}u(T(y,s)),$$
$$\tilde{\beta}(y) = \beta (T(y,s)) ,$$
$$\tilde{F}(X,y)= M^{-k}F(M^{k}X,T(y,s)) $$ and $$\tilde{f}(y)= M^{-k}f(T(y,s)) .$$
Then $\tilde{u}$ is a viscosity solution of
\begin{align}  \label{paob_sisc}
\left\{ \begin{array}{ll}
\tilde{F}(D^{2}\tilde{u}, y,s) - \tilde{u}_{t} = \tilde{f} & \textrm{in $ r \Omega(0,0)$,}\\
\tilde{\beta} \cdot D\tilde{u} = 0 & \textrm{on $rS (0,0)$,}\\
\end{array} \right. 
\end{align}
since $(r/2^{i})\Omega(x_{1}',0, t^{\ast}) \subset r \Omega(x_{0},t_{0})$.
Observe that $\tilde{\beta} \in C^{2}(rS (0,0)) $ and
$\tilde{F}$ has the same elliptic constant of $F$.
Let $$\psi_{\tilde{F}}((y,s),(0,0))=\psi_{F}(T(y,s),(x_{1}',0,t^{\ast})).$$ 
Then we also have
$ ||\psi_{\tilde{F}}||_{L^{n+1}( r\Omega' (0,0))} \le C \epsilon $ for some $C=C(n)>0$

In addition, we obtain 
\begin{align*}
||\tilde{f}||_{L^{n+1}(   r \Omega(0,0))}& = \bigg(  \int_{   r \Omega(0,0)}    | \tilde{f}( y,s)|^{n+1} \ dy ds \bigg)^{\frac{1}{n+1}}
\\ & \le C(n)c_{0}
\\  & \le \epsilon
\end{align*}
by using Proposition \ref{mppara} and choosing $c_{0}$ small enough.

 
One the other hand, we have
$$  T^{-1}\overline{K}^{1} \cap  G_{1}(\tilde{u},T^{-1}( r \Omega(0,0))) \neq \varnothing $$
by the assumption $  \overline{K}^{1} \cap G_{M^{k}}(u,\Omega)  \neq \varnothing$.
And since $|x_{1}-\tilde{x}_{1}| < r\sqrt{n}/2^{i} $, we observe that $ |T^{-1}\tilde{x_{1}}| < 9r\sqrt{n}$. 
Consequently, applying Lemma \ref{pl02} to $\tilde{u}$, we get
$$ \frac{| T^{-1}K \cap  G_{M}(\tilde{u},T^{-1} (r \Omega(0,0))) |}{|T^{-1}K|} \ge 1- \epsilon_{0}.$$
Then it follows immediately that
$$ \frac{|  K \cap  G_{M^{k+1}}(u, r \Omega(0,0)) |}{|K|} \ge 1- \epsilon_{0}.$$
This leads to a contradiction.

Now we consider the interior case  $x_{1, n} > 8r\sqrt{n} / 2^{i}$. 
Observe that $$ Q_{8r\sqrt{n}/2^{i}}(x_{1}+re_{n}/2^{i+1},t_{1}) \subset Q_{8r\sqrt{n}h_{1}^{-1},8r\sqrt{n}}^{+}(x_{0},t_{0})$$ in this case.
Again, set $T:Q_{8r\sqrt{n}} \to  Q_{8r\sqrt{n}/2^{i}}(x_{1}+re_{n}/2^{i+1},t_{1}) $ such that
$$ T(y,s) = \bigg(x_{1}+\frac{re_{n}}{2^{i+1}} + \frac{y}{2^{i+1}}, t_{1}+\frac{s}{2^{2(i+1)}}\bigg).$$
and we write $$ \tilde{u}(y,s)=  2^{2(i+1)}M^{-k}u(T(y,s)),$$ 
$$\tilde{F}(X,y,s)= M^{-k}F(M^{k}X,T(y,s)) $$ and $$\tilde{f}(y,s)= M^{-k}u(T(y,s)) .$$ 
We can check that $\tilde{u}$ is a solution of  $$ \tilde{F}(D^{2}\tilde{u}, y, s) -\tilde{u}_{t} = \tilde{f}(y,s)  \qquad \textrm{in} \  Q_{8r \sqrt{n}} $$
in the viscosity sense.
Applying \cite[Corollary 5.2]{MR1135923} to  $\tilde{u}$, we can also deduce our desired result.
\end{proof}

\begin{proof}[Proof of Theorem \ref{paob_kthm}]
We fix $(x_{0},t_{0}) \in Q_{2/3}^{+} \cup Q_{2/3}^{\ast}$.
If $(x_{0},t_{0}) \in Q_{2/3}^{\ast}  $, let $r$ be a fixed number in $ \big(0, \min \big\{\frac{1-|x_{0}|}{14\sqrt{n}}h_{1}, \sqrt{-\frac{t_{0}}{15} }\big\}\big) $ and we set
$$ K = \frac{\epsilon r^{\frac{n+2}{n+1}} }{\epsilon r^{-1} ||u||_{ L^{\infty}( r \Omega(x_{0},t_{0}) )} +||f||_{ L^{n+1}(  r \Omega(x_{0},t_{0}) )}} .$$ 
Here, $  \Omega= B_{14 \sqrt{n}h_{1}^{-1},14 \sqrt{n}}^{+} \times (0,15] $ with $h_{1}=h_{1}(\delta_{0})$ as in Lemma \ref{pl02} and $\epsilon =\epsilon (n, \epsilon_{0}, \lambda, \Lambda, p, \delta_{0},  ||\beta||_{C^{2}(\overline{Q}_{1}^{\ast})} ) $ is a constant as in Lemma \ref{pl02} with $ \epsilon_{0}  \in (0,1)$ to be chosen later. 

Let $$ \tilde{u}(y,s)=Kr^{-2}u(ry+x_{0},r^{2}s+t_{0}),$$ 
$$ \tilde{f}(y,s)=Kf(ry+x_{0},r^{2}s+t_{0}) ,$$
$$\tilde{\beta} (y,s)=\beta (ry+x_{0},r^{2}s+t_{0}),$$
 and $$\tilde{F}(X,y)=KF(K^{-1}X,ry+x_{0},r^{2}s+t_{0}) . $$
Then, $\tilde{u}$ is a  solution of 
\begin{align}  
\left\{ \begin{array}{ll}
\tilde{F}(D^{2}\tilde{u}, y,s) - \tilde{u}_{t} = \tilde{f} & \textrm{in $  \Omega$,}\\
\tilde{\beta} \cdot D\tilde{u} = 0 & \textrm{on $S :=T_{14 \sqrt{n}h_{1}^{-1}}^{+} \times (0,15] $}\\
\end{array} \right. 
\end{align} in the viscosity sense.
It can be checked without difficulty that $F$ and $\tilde{F}$ have the same elliptic constants, 
$\tilde{\beta} \in C^{2}(S) $, 
 $ ||\tilde{u}||_{ L^{\infty}( \Omega )} \le 1$,
$$ ||\psi_{\tilde{F}}||_{ L^{n+1}( \Omega )} \le C(n)\epsilon_{0} \le \epsilon,$$
$$ ||\tilde{f}||_{ L^{n+1}( \Omega )} \le Kr^{-\frac{n+2}{n+1}}||f||_{ L^{n+1}( r\Omega(x_{0},t_{0}) )} \le \epsilon$$
for a sufficiently small $\epsilon_{0}$.
Thus, the assumption of Lemma \ref{paob_lem3} is satisfied.  
Set 
$$ \alpha_{k}= |A_{M^{k}} (u,\Omega ) \cap (K_{1}^{n-1} \times (0,1) \times (2,3))| ,$$
$$ \beta_{k}   = | \{ (x,t) \in K_{1}^{n-1} \times (0,1) \times (2,3) : M(|f|^{n+1}) (x,t)\ge (c_{0}M^{k})^{n+1} \}|  $$ and choose $\epsilon_{0}=1/(4M^{p}) $.
By direct calculation, we have 
$$ \alpha_{k} \le (2\epsilon_{0})^{k} + \sum_{i=0}^{k-1} (2\epsilon_{0})^{k-i}\beta_{i}.$$ 
We also observe that
$$||M(|f|^{n+1})||_{L^{\frac{p}{n+1}}} \le C(n,p) $$ by Proposition \ref{mppara}, and this implies
$$ \sum_{i=0}^{\infty}M^{pk}\alpha_{k} \le C(n,p).$$
Using Proposition \ref{paralpeq}, we discover
$$   ||\tilde{u}_{t}||_{ L^{p}( Q_{\frac{1}{2}}^{+}(0,-\frac{1}{8} ))} + ||D^{2}\tilde{u}||_{ L^{p}( Q_{\frac{1}{2}}^{+}(0,-\frac{1}{8}) )} \le C,$$
that is,
$$ ||u_{t}||_{ L^{p}( Q_{r/2}^{+}( x_{0},t_{0}-\frac{r^{2}}{8}))}+  ||D^{2}u||_{ L^{p}( Q_{r/2}^{+}(x_{0},t_{0}-\frac{r^{2}}{8}) )}\hspace{-0.15em} \le \hspace{-0.15em} C(||u||_{ L^{\infty}( Q_{1}^{+} )}+||f||_{ L^{p}( Q_{1}^{+} )} ),$$
where $C = C(n,  \lambda, \Lambda, p, r, \delta_{0}, ||\beta||_{C^{2}(\overline{Q}_{1}^{\ast})} ) > 0$.
 
Besides, when $(x_{0},t_{0}) \in Q_{2/3}^{+}  $, we can apply the results of interior estimates, like as in \cite[Theorem 5.6]{MR1135923}.
Combining the interior and boundary estimates, we get
$$   ||u_{t}||_{ L^{p}( Q_{\frac{1}{2}}^{+}(0,-\frac{1}{8} ))} + ||D^{2}u||_{ L^{p}( Q_{\frac{1}{2}}^{+}(0,-\frac{1}{8}) )} \le C,$$
where $C = C(n,  \lambda, \Lambda, p, \delta_{0},  ||\beta||_{C^{2}(\overline{Q}_{1}^{\ast})} ) > 0$.

We also need to establish proper regularity results in $Q_{\frac{1}{2}}^{+} \times [-1/8, 0) $.
For these estimates, we extend $F$ and $\beta $ such that our assumptions are satisfied.
Then we can obtain the estimate \eqref{mainest_po}.
\end{proof}
 
\section{Boundary $W^{1,p}$-estimates and global estimates}
We have obtained $W^{2,p}$-regularity for solutions of \eqref{paob_eq} in the previous section.
Here, we extend this regularity to the case when the function $F$ also contains ingredients  $q$ and $r$.
Again, we consider an oscillation function $\psi_{F}$ as
$$ \psi_{F}((x,t), (y,s)) := \sup_{X \in S(n) \backslash \{0\}} \frac{|F(X,0,0,x,t)-F(X,0,0,y,s)|}{||X||},$$
defined as \eqref{defposc}.

Let $u$ be a viscosity solution of the following problem
\begin{align}  \label{paob_inq1}
\left\{ \begin{array}{ll}
F(D^{2}u, Du, u,  x, t) - u_{t} = f & \textrm{in $Q_{1}^{+}$,}\\
\beta \cdot Du = 0 & \textrm{on $Q_{1}^{\ast}$,}\\
\end{array} \right. 
\end{align}
and assume that this $u$ also solves 
\begin{align}  
\left\{ \begin{array}{ll}
F(D^{2}u, 0, 0,  x,t)-u_{t} = \tilde{f} & \textrm{in $Q_{1}^{+}$,}\\
\beta \cdot Du = 0 & \textrm{on $Q_{1}^{\ast}$}\\
\end{array} \right.
\end{align}
for some function $\tilde{f}$ in the viscosity sense.

We already know that $W^{2,p} $-norm of $u$ is estimated by $L^{\infty} $-norm of $u$ and $L^{n+1} $-norm of $\tilde{f}$ by Theorem \ref{paob_kthm}.
And we also have 
$$ | \tilde{f}| \le |f| + b |Du| +c |u| $$
by virtue of the structure condition \eqref{paob_sc}.
Thus, we need to obtain $W^{1,p}$-regularity for $u$ in order to reach our goal.
The following theorem provides the type of estimates which we want to derive.

\begin{theorem}  \label{paobw1p}
 Let $n+1<p < \infty$.
Assume that $F$ satisfies the structure condition \eqref{paob_sc} with $F(0,0,0,x,t)=0$ and $u$ be a viscosity solution of
\eqref{paob_inq1}
where $ f \in L^{p}(Q_{1}^{+})\cap C({\overline{Q}_{1}^{+}})$ and $\beta \in C^{0,\overline{\alpha}}(\overline{Q}_{1}^{\ast}) $ for some $0<\overline{\alpha}<1$ depending on $n,\lambda,\Lambda$ and $\delta_{0}$.
Then, there exists a constant $\epsilon_{0}=\epsilon_{0}(n,\lambda,\Lambda,p,\delta_{0},\alpha) $ such that if
$$ \bigg(  \kint_{Q_{r}(x_{0},t_{0}) \cap Q_{1}^{+}}     \psi( (x_{0},t_{0}),(x,t))^{p} \ dx dt \bigg)^{1/p} \le \epsilon_{0}$$
for any $(x_{0},t_{0}) \in Q_{1}^{+} $ and $r \le r_{0}$ for some $r_{0}>0$, then
 $u \in C^{1, \alpha}(\overline{Q}_{\frac{1}{2}}^{+})$ with $ \alpha = \alpha (n, p, \lambda, \Lambda, \delta_{0})\in (0,\overline{\alpha})$ and we have the estimate
\begin{align} \label{pake1} ||u||_{C^{1,\alpha}(\overline{Q}_{\frac{1}{2}}^{+})} \le C (||u||_{L^{\infty}(Q_{1}^{+})}+ ||f||_{L^{p}(Q_{1}^{+})})
\end{align}
for some $C= C(n,   \lambda, \Lambda , b,c, p, ||\beta||_{C^{0,\overline{\alpha}}(\overline{Q}_{1}^{\ast})} ,r_{0})$.
\end{theorem}

We now introduce a useful building block in this section.
One can find its proof in \cite[Theorem 6.1]{MR1789919}.

\begin{proposition} \label{paobappr}
For $k \in \mathbb{N}$, let $ \Omega_{k} \subset \Omega_{k+1} $ be an increasing sequence of domains in $\mathbb{R}^{n} \times \mathbb{R}$ and $\Omega : = \cup_{k \ge 1} \Omega_{k} $.
Let  $p>n+1$ and $F, F_{k}$ be continuous and measurable in $x$ and $t$, and satisfy structure condition \eqref{paob_sc}.
Assume that $f \in L^{p}(\Omega ) $, $ f_{k} \in L^{p}(\Omega_{k} )$ and that $ u_{k } \in C(\Omega_{k} )$ are viscosity subsolutions (supersolutions, respectively) of $$F_{k}(D^{2}u_{k}, Du_{k}, u_{k}, x, t)- (u_{k})_{t} = f_{k} \quad \textrm{in} \ \Omega_{k} .$$
Suppose that $u_{k} \to u$ locally uniformly in $\Omega  $ and that for any cylinders $Q_{r}(x_{0},t_{0}) \subset \Omega  $ and $ \varphi \in C^{2}(Q_{r}(x_{0},t_{0}))$,
\begin{align} \label{fccon}  ||(s-s_{k})^{+} ||_{L^{p}(Q_{r}(x_{0},t_{0}))} \to 0 \qquad  \big( ||(s-s_{k})^{-} ||_{L^{p}(Q_{r}(x_{0},t_{0}))} \to 0   \big)
\end{align}
where $$ s(x,t) = F(D^{2}\varphi, D \varphi, u, x,t) -f(x,t) ,$$  $$ s_{k}(x,t) = F(D^{2}\varphi_{k}, D \varphi_{k}, u_{k}, x,t) -f_{k}(x,t) .$$
Then $u$ is a viscosity subsolution (supersolution) of $$  F(D^{2}u, Du, u, x,t)-u_{t}=f \ \textrm{ in} \  \Omega  .$$
\end{proposition}

We first prove a compactness lemma for problems with oblique boundary data.
\begin{lemma}   \label{parobw1plem}
Let $n+1 <p<\infty$ and $ 0 \le \nu \le 1$. 
Assume that $F$ satisfies \eqref{paob_sc} with $F(0,0,0,x,t)\equiv0$ and $\beta \in C^{2}(Q_{2}^{\ast} )$ with $\beta \cdot \mathbf{n} \ge \delta_{0}$ for some $\delta_{0} > 0$.
Then, for every $\rho >0$,  $\varphi \in C(\partial_{p} Q_{1}^{\nu} (0', \nu,0)) $ with $ ||\varphi||_{L^{\infty}(\partial_{p} Q_{1}^{\nu}(0', \nu,0))} \le C_{1}$ for some $ C_{1}>0$ and $ g \in  C^{0,\alpha}(\overline{Q}_{2}^{\ast})$ with $0< \alpha <1 $ and $  ||g||_{ C^{0,\alpha}(\overline{Q}_{2}^{\ast})} \le C_{2}$ for some $C_{2}>0 $,
 there exists a positive number $\delta=\delta(\rho,n,\lambda,\Lambda,\delta_{0},p,C_{1},C_{2})<1$ such that if
$$ ||\psi((0,0), (\cdot,\cdot))||_{L^{p}(Q_{2}^{\nu}(0', \nu,0))}+ ||f||_{L^{p}(Q_{2}^{\nu}(0', \nu,0))}+ b+ c \le \delta, $$
then for any $u$ and $v$ solving
\begin{align*} 
\left\{ \begin{array}{ll}
F(D^{2}u, Du, u,  x, t) - u_{t}= f & \textrm{in $Q_{1}^{\nu}(0', \nu,0)$,}\\
u= \varphi  & \textrm{on $ \partial_{p} Q_{1}^{\nu}(0', \nu,0) \backslash  Q_{1}^{\ast}$,}\\
\beta \cdot Du = g & \textrm{on $Q_{1}^{\ast}$,}\\
\end{array} \right.
\end{align*}
and
\begin{align*} 
\left\{ \begin{array}{ll}
F(D^{2}v, 0, 0,  0,0)-v_{t} =0 & \textrm{in $Q_{\frac{3}{4}}^{\nu}(0', \nu,0)$,}\\
v = u & \textrm{on $ \partial_{p}  Q_{\frac{3}{4}}^{\nu}(0', \nu,0) \backslash  Q_{\frac{3}{4}}^{\ast}$,}\\
\beta \cdot Du = g & \textrm{on $Q_{\frac{3}{4}}^{\ast} $}\\
\end{array} \right.
\end{align*}in the viscosity sense, respectively,
we have $||u-v||_{L^{\infty}(Q_{\frac{3}{4}}^{\nu}(0', \nu,0) )} \le \rho $.
\end{lemma}

\begin{proof} 
Assume that there is a number $\rho_{0} >0$ such that
if $u_{k}$ and $v_{k}$ solve
\begin{align} 
\left\{ \begin{array}{ll}
F_{k}(D^{2}u_{k}, Du_{k}, u_{k},  x, t) - (u_{k})_{t}= f_{k} & \textrm{in $Q_{1}^{\nu_{k}}(0', \nu_{k},0)$,}\\
u_{k}= \varphi_{k}  & \textrm{on $ \partial_{p} Q_{1}^{\nu_{k}}(0', \nu_{k},0) \backslash  Q_{1}^{\ast}$,}\\
\beta \cdot Du_{k} = g_{k} & \textrm{on $Q_{1}^{\ast}$,}\\
\end{array} \right.
\end{align}
and
\begin{align} 
\left\{ \begin{array}{ll}
F_{k}(D^{2}v_{k}, 0, 0, 0, 0) - (v_{k})_{t}=0 & \textrm{in $Q_{\frac{3}{4}}^{\nu_{k}}(0', \nu_{k},0)$,}\\
v_{k} = u_{k} & \textrm{on $  \partial_{p}  Q_{\frac{3}{4}}^{\nu_{k}} (0', \nu_{k},0) \backslash  Q_{\frac{3}{4}}^{\ast} $,}\\
\beta \cdot Dv_{k} = g_{k} & \textrm{on $Q_{\frac{3}{4}}^{\ast}$}\\
\end{array} \right.
\end{align} in the viscosity sense, respectively,
then $||u_{k}-v_{k}||_{L^{\infty}(Q_{\frac{3}{4}}^{\nu_{k}}(0', \nu_{k},0))}    >  \rho_{0} $
 for any $F_{k}, f_{k}$, $  b_{k}, c_{k}, \psi_{F_{k}}$
with
$$  ||\psi_{F_{k}}((0,0), (\cdot,\cdot))||_{L^{p}(Q_{2}^{\nu_{k}}(0', \nu_{k},0))}, ||f_{k}||_{L^{p}(Q_{2}^{\nu_{k}}(0', \nu_{k},0))}, b_{k}, c_{k} \le  \delta_{k} \to 0  $$
as $k \to \infty$.
Furthermore, we also assume that $\varphi_{k} \in C(\partial_{p} Q_{1}(0', \nu_{k},0))$ with $ ||\varphi_{k}||_{L^{\infty}(\partial_{p} Q_{1}(0', \nu_{k},0))} \le C_{1}$
and  $g_{k} \in  C^{0,\alpha} (  \overline{Q}_{2}^{\ast} )$ with
$  ||g_{k}||_{ C^{0,\alpha}(\overline{Q}_{2}^{\ast} )} \le C_{2}$ for each $k$,
 respectively.  

From the structure condition \eqref{paob_sc}, 
we can find a subsequence $F_{k_{i}}$ and a function $F_{\infty}$ so that
$ F_{k_{i}}(\cdot, \cdot, \cdot, 0,0) $ converges uniformly to $F_{\infty}(\cdot)$ 
 on compact subsets of $S(n) \times \mathbb{R}^{n} \times \mathbb{R} $
by using Arzel\'{a}-Ascoli theorem.
Then for any $  \delta_{1} \in(0, 1)$ and sufficiently large $k$, it follows from  Lemma \ref{abptob} that 
\begin{align*}& ||u_{k}||_{L^{\infty}(Q_{1}^{\nu_{k}} )} 
\le
||\varphi_{k}||_{L^{\infty}( \partial_{p} Q_{1}^{\nu_{k} }(0', \nu_{k},0)\backslash  Q_{1}^{\ast})}
\\ &  +C(n, \lambda, \Lambda, \delta_{0})(||f_{k}||_{L^{n+1}(Q_{1}^{\nu_{k}}(0', \nu_{k},0))}\hspace{-0.2em}+\hspace{-0.2em}||g_{k}||_{L^{\infty}(Q_{1}^{\ast})}
\hspace{-0.2em}+\hspace{-0.2em} c_{F_{k}}||u_{k}||_{L^{\infty}(Q_{1}^{\nu_{k}}(0', \nu_{k},0) )} )
\end{align*}
and this implies
\begin{align*}
||u_{k}||_{L^{\infty}(Q_{1}^{\nu_{k}} (0', \nu_{k},0))} \le C(C_{1},C_{2},n, \lambda, \Lambda, \delta_{0} ).
\end{align*} 
Moreover, applying Lemma \ref{parobhol} to $u_{k}$, we have 
\begin{align}  \begin{split} \label{paobw1pholest} &||u_{k}||_{C^{0,\alpha_{1}}( Q_{1, \delta_{1}}^{\nu_{k}}(0', \nu_{k},0))} \\ & \le  C (||u_{k}||_{L^{\infty}(Q_{1}^{\nu_{k}}(0', \nu_{k},0))}+||f_{k}||_{L^{n+1}(Q_{1}^{\nu_{k}}(0', \nu_{k},0))}+||g_{k}||_{L^{\infty}( Q_{1}^{\ast})}) \delta_{1}^{-\alpha_{1}} \\ &   \le  C( C_{1}, C_{2}, n, \lambda, \Lambda, \delta_{0} ) \delta_{1}^{-\alpha_{1}}, \end{split}
\end{align}
where   $\alpha_{1} \in (0,1)$ only depends on $n, \lambda, \Lambda$ and $ \delta_{0}$.
Then we obtain 
\begin{align}  \label{pow1phe2}
||u_{k}||_{C^{0,\alpha_{1}}( Q_{\frac{7}{8}}^{\nu_{k}}(0', \nu_{k},0))}   \le  C( C_{1}, C_{2}, n, \lambda, \Lambda, \delta_{0} )  
\end{align}
from \eqref{paobw1pholest}. 

Now assume that there is a subsequence $ \{ \nu_{k_{i}} \} \subset \{ \nu_{k} \}$ and a number $0 \le \nu_{\infty} \le 1$
such that $ \nu_{k_{i}} \to \nu_{\infty}$ as $ i \to \infty $.
It is sufficient to consider the case of monotone subsequences.
When $\{\nu_{k_{i}} \} $ is decreasing, then $Q_{7/8}^{\nu_{\infty}}(0', \nu_{\infty},0) \subset Q_{7/8}^{\nu_{k}}(0', \nu_{k},0)$ for every $i$.
Thus we can observe that there is a function $ u_{\infty}$ such that 
$u_{k_{i}}$ converges uniformly to $u_{\infty}$ on $Q_{7/8}^{\nu_{\infty}}(0', \nu_{\infty},0)$
by using Arzel\'{a}-Ascoli theorem. 
For increasing subsequences, we can see that $Q_{7/8}^{\nu_{\infty}}(0', \nu_{\infty},0) \subset Q_{15/16}^{\nu_{k}}(0', \nu_{k},0) $ 
for sufficiently large $k$.
Then we can also deduce the uniform convergence for increasing subsequences.
Therefore, there exists a subsequence $\{ u_{k_{i}} \} \subset \{ u_{k} \}$ and a function $u_{\infty}$ with
$ u_{k_{i}}$ converging uniformly  to $ u_{\infty}$ 
in $  Q_{\frac{7}{8}}^{\nu_{\infty}} (0', \nu_{\infty},0) $.

Take a test function $\phi \in C^{2}(\overline{Q}_{\frac{7}{8}}^{\nu_{\infty}})$. Then we can observe that
\begin{align*}
| F_{k_{i}} (D^{2}\phi, D\phi,& u_{k_{i}}, x,t) - f_{k_{i}}(x,t) - F_{\infty} (D^{2} \phi, 0,0,0,0)| \\
& \le  c_{k_{i}}C(C_{1}) + b_{k_{i}} |D\phi| + \psi_{F_{k_{i}}}((0,0),(x,t))|D^{2}\phi| \\ & \qquad + |f_{k_{i}}|+|( F_{k_{i}}-F_{\infty})(D^{2} \phi, 0,0,0,0)|
\end{align*}
and this implies
$$ \lim_{i \to \infty}|| F_{k_{i}} (D^{2}\phi, D\phi, u_{k_{i}}, x,t) - f_{k_{i}}(x,t) - F_{\infty} (D^{2} \phi, 0,0,0,0)||_{L^{p}(Q_{r}(x_{0},t_{0}))} = 0 $$ 
for any $Q_{r}(x_{0},t_{0}) \subset Q_{\frac{7}{8}}^{\nu_{\infty}} (0', \nu_{\infty},0)$. 
Since $\{ g_{k} \} \subset  C^{0,\alpha}(\overline{Q}_{1}^{\ast}) $ are uniformly bounded and equicontinuous on $Q_{1}^{\ast}$,
we can find a function $ g_{\infty} \in  C^{0,\alpha}(Q_{1}^{\ast})$ by Arzel\'{a}-Ascoli theorem.
Thus by Proposition \ref{paobappr} and \cite[Proposition 31]{cm2019oblique}, we have
\begin{align} 
\left\{ \begin{array}{ll}
F_{\infty}(D^{2}u_{\infty}, 0, 0,0, 0) - (u_{\infty})_{t} = 0 & \textrm{in $Q_{\frac{7}{8}}^{\nu_{\infty}}(0', \nu_{\infty},0)$,}\\
 \beta \cdot Du_{\infty} = g_{\infty} & \textrm{on $  Q_{\frac{7}{8}}^{\ast} $}\\
\end{array} \right.
\end{align}
in the viscosity sense. 
 
Set $w_{k_{i}}:=u_{\infty}-v_{k_{i}}$. Then we observe that
\begin{align} 
\left\{ \begin{array}{ll}
w_{k_{i}} \in S (\lambda/n, \Lambda, 0) & \textrm{in $Q_{ \frac{3}{4}}^{\nu_{\infty}}(0', \nu_{\infty},0)$,}\\
w_{k_{i}}=u_{\infty}-u_{k_{i}} & \textrm{on $ \partial_{p} Q_{ \frac{3}{4}}^{\nu_{\infty}}(0', \nu_{\infty},0) \backslash  Q_{\frac{3}{4}}^{\ast}  $,} \\
\beta \cdot Dw_{k_{i}} = g_{\infty} -g_{k_{i}} & \textrm{on $ Q_{\frac{3}{4}}^{\ast}$}\\
\end{array} \right. 
\end{align}
in the viscosity sense by means of Lemma \ref{cm11}.
Applying Lemma \ref{parobabpt} to $w_{k_{i}}$, we get 
\begin{align} \label{pokest}\begin{split}  ||w_{k_{i}}&||_{L^{\infty}(Q_{ \frac{3}{4}}^{\nu_{\infty}}(0', \nu_{\infty}))} \\ &
\le ||u_{\infty}-u_{k_{i}}||_{L^{\infty}( \partial_{p} Q_{\frac{3}{4}}^{\nu_{\infty}}(0', \nu_{\infty},0)  \backslash  Q_{\frac{3}{4}}^{\ast} ))}\\ & \qquad+ C(n,\lambda,\Lambda,\delta_{0}) ||g_{\infty}-g_{k_{i}}||_{L^{\infty}(Q_{\frac{3}{4}}^{\nu_{\infty}}(0',\nu_{\infty},0))}
\end{split}
\end{align} and this converges to zero
as $i \to \infty$.
This shows that $v_{k_{i}} $ converges uniformly to $u_{\infty}$ on $ Q_{ \frac{3}{4}}^{\nu_{\infty}}(0',\nu_{\infty},0) $. 
But it is a contradiction since we already have assumed $||u-v||_{L^{\infty}(Q_{\frac{3}{4}}^{\nu_{\infty}}(0',\nu_{\infty},0))}    >  \rho_{0} $.
\end{proof}

\begin{proof}[Proof of Theorem \ref{paobw1p}]

Let $  p' \in (n+1 ,p)$, $(y,s) \in Q_{\frac{1}{2}}\cap \{ x_{n} \ge 0 \}$, $d = \min \{ {\frac{1}{2}, r_{0}} \}$,
and we choose $\sigma > 0$ such that
$$ \sigma \le \frac{d}{2}, \quad \sigma b \le \frac{\delta}{32MC(n)}, \quad \sigma^{2}c \le \frac{\delta}{32(M+1)C(n)}.  $$
Here, $\delta$ is the same as in Lemma \ref{parobw1plem}, $ C(n)$ is universal and $M$ will be chosen later. 

We first consider the case $ y_{n} < \sigma /2 $. Define
\begin{align*} K &= K(y,s) 
\\ &  = \hspace{-0.15em}  ||u||_{L^{\infty}(Q_{d}(y,s) \cap Q_{1}^{+})}  \hspace{-0.12em} + \hspace{-0.12em}  \frac{1}{\epsilon_{0}} \sup_{r \le d}  \bigg[r^{1-\alpha} \bigg(\hspace{-0.3em}  r^{-(n+2)}\hspace{-0.3em} \int_{Q_{r}(y,s) \cap Q_{1}^{+}} \hspace{-0.5em}   | f( x,t)|^{p'}  dxdt  \bigg)^{\frac{1}{p'}} \bigg],
\end{align*}
where $0< \alpha<1$ is  to be determined. 
Observe that for any $(y,s)$,
$$ K(y,s) \le  ||u||_{L^{\infty}(Q_{1}^{+})} + C(n, \epsilon_{0}) \big[ M(f^{p})(y,s) \big]^{\frac{1}{p}} < \infty .$$  
Let $$ \tilde{u}(x,t) = \frac{1}{K} u(\sigma x +y, \sigma^{2}t+s) ,$$ 
$$\tilde{f}(x,t) =  \frac{ \sigma^{2}}{K}  f(\sigma x +y, \sigma^{2}t+s),$$
$$ \tilde{\beta}(x,t) = \beta (\sigma x+y, \sigma^{2}t+s),$$  
$$ \tilde{F}(X,p,r,x,t) = \frac{\sigma^{2} }{K} F(K\sigma^{-2}X, K\sigma^{-1}p, Kr, \sigma x+y, \sigma^{2}t+s)$$
and $ \nu = y_{n}/\sigma $.
Then $\tilde{u}$ is a viscosity solution of 
\begin{align} 
\left\{ \begin{array}{ll}
\tilde{F}(D^{2}\tilde{u}, D\tilde{u}, \tilde{u},  x,t ) - \tilde{u}_{t} = \tilde{f} & \textrm{in $Q_{2}^{\nu}$,}\\
\tilde{\beta} \cdot D\tilde{u} = 0 & \textrm{on $ Q_{2} \cap \{x_{n}= -\nu \}$.}\\
\end{array} \right.
\end{align}
We can see that $\tilde{F}$ also satisfies \eqref{paob_sc} with $b_{\tilde{F}}=\sigma b $, $c_{\tilde{F}}= \sigma^{2} c$,
$$ r^{1-\alpha} \bigg( r^{-(n+2)} \int_{ Q_{r}^{\nu}}    | \tilde{f}( x,t)|^{p'} \ dx dt  \bigg)^{\frac{1}{p'}} \le \epsilon_{0} \sigma^{1+\alpha}$$ for every $ r \in (0,2)$ 
and
$$ ||\psi_{\tilde{F}} ((0,0), (\cdot,\cdot))||_{L^{p'}(Q_{1}^{\nu})} \le \delta $$ for small $\epsilon_{0}$. 

Now we establish $C^{1, \alpha}$-regularity.
To this end, we need to show that there are some universal constants $\mu, C_{1},C_{2} >0 $, $ 0 < \alpha, \alpha_{0} <1$ and linear functions $ l_{k,s}(x)=a_{k,s}+b_{k,s}\cdot x$ for each $ k \ge -1 $ such that 
\begin{itemize}
\item[(i)] $ || \tilde{u}-l_{k,s}||_{L^{\infty}(Q_{\mu^{k}}^{\nu})} \le \mu^{k(1+\alpha)} $. \\
\item[(ii)] $ |a_{k-1,s}-a_{k.s}|+\mu^{k-1} |b_{k-1.s}-b_{k,s}|  \le 2C_{2}\mu^{(k-1)(1+\alpha)} $.
\item[(iii)]  $\frac{|(\tilde{u}-l_{k,s})(\nu^{k}x_{1},\mu^{2k}t_{1})-(\tilde{u}-l_{k,s})(\nu^{k}x_{2},\mu^{2k}t_{2}) |}{(|x_{1}-x_{2}|+|t_{1}-t_{2}|^{\frac{1}{2}})^{\alpha_{0}}} \le 3C_{1}
\mu^{k(1+\alpha)}$
\\ for every $(x_{1},t_{1}),(x_{2},t_{2})\in Q_{1}\cap \{x_{n}\ge -\nu \}$.
\item[(iv)] $||\beta_{k}\cdot b_{k,s}||_{C^{0,\overline{\alpha}}(Q_{1} \cap \{ x_{n}=-\nu/\mu^{k}\})} \le \mu^{k\alpha}$.
\end{itemize}
Here, $C_{1}=C(n, \lambda, \Lambda,\delta_{0})$, $\alpha_{0}=\alpha(n, \lambda, \Lambda,\delta_{0})$
where $C, \alpha$ are constants as in Lemma \ref{parobhol} when it is applied to $\tilde{u} \in S^{\ast}(\lambda/n, \Lambda, 1 , \tilde{f})$ in $Q_{2}^{\nu}u,0)$,
and $\beta_{k}$ is to be defined later.  

Let $ l_{-1,s}= l_{0,s} =0$ and take $0<\alpha <\alpha_{0}$.
Consider a fixed number $\mu \le 1/4 $ such that 
\begin{align} \label{mucond} 8C_{2}(1+||\beta||_{C^{0,\overline{\alpha}}(\overline{Q}_{1}^{\ast})}) \mu^{1+\overline{\alpha}} \le \mu^{1+\alpha}
\end{align}  and
\begin{align} \label{defm}
M=4C_{1} \sum_{i=0}^{\infty} \bigg( \frac{1}{4} \bigg)^{i \alpha} \ge 4C_{1} \sum_{i=0}^{\infty} \mu^{i \alpha}.
\end{align} 

We use induction to prove that the above conditions are satisfied for every $k$.
It can be checked without difficulty when $k=0$.
Next we show that (i)-(iv) are still satisfied for $k+1$
under the assumption that these conditions hold for $k>0$. 

Let
$$ v_{k}(x,t) = \frac{(\tilde{u}-l_{k,s})(\mu^{k}x, \mu^{2k}t)}{\mu^{k(1+\alpha)}} .$$
Then $v_{k}$ is a viscosity solution of  
\begin{align} 
\left\{ \begin{array}{ll}
F_{k}(D^{2}v_{k}, Dv_{k}, v_{k},  x,t)- (v_{k})_{t} = f_{k}+g_{k} & \textrm{in $Q_{2}^{\frac{\nu}{\mu^{k}}}$,}\\
\beta_{k} \cdot Dv_{k} = -(\beta_{k} \cdot b_{k,s})/\mu^{k \alpha} & \textrm{on $Q_{2} \cap \{x_{n}= -\frac{\nu}{\mu^{k}} \} $,}\\
\end{array} \right.
\end{align}
where $$F_{k}(X,q,r,x,t)=\mu^{k(1-\alpha)}\tilde{F}(\mu^{k(\alpha-1)}X,\mu^{k\alpha}q, \mu^{k(\alpha+1)}r, \mu^{k}x , \mu^{2k}t),$$ 
\begin{align*} 
g_{k}(x,t)=&F_{k}(D^{2}v_{k},Dv_{k},v_{k},x,t) \\ & 
-F_{k}(D^{2}v_{k}, Dv_{k}+\mu^{-k \alpha}b_{F_{k}}, v_{k}+\mu^{-k(1+\alpha)}l_{k,s}(\mu^{k}x ),x ,t), 
\end{align*} 
$$ f_{k}(x,t)= \mu^{k(1-\alpha)} \tilde{f} (\mu^{k}x, \mu^{2k}t)$$
and
$$\beta_{k}(x,t)=\beta(\mu^{k}x,\mu^{2k}t) .$$ 
We remark that $\psi_{F_{k}}((0,0),(x,t)) = \psi_{\tilde{F}}(  (0,0),  (\mu^{k}x, \mu^{2k}t)) $ and
$ \tilde{F}$ satisfies \eqref{paob_sc} with $b_{F_{k}}= \mu^{k} b_{\tilde{F}}  $ and $c_{F_{k}}=\mu^{2k} c_{\tilde{F}} $. 

Now we can observe that
\begin{align*}
|g_{k}(x,t)| \le  b_{F_{k}} \cdot \mu^{-k \alpha} M + c_{F_{k}} \cdot \mu^{-k (\alpha+1)}M \le \mu^{k(1-\alpha)} \frac{\delta}{16}
\end{align*}
and this implies 
\begin{align*}
 ||f_{k}+g_{k}||_{L^{p'}(Q_{1}^{\frac{\nu}{\mu^{k}}})} &  \le
 ||f_{k}||_{L^{p'}(Q_{1}^{\frac{\nu}{\mu^{k}}})} +||g_{k}||_{L^{p'}(Q_{1}^{\frac{\nu}{\mu^{k}}})} 
\\ & \le \frac{\delta}{2} + \frac{\delta}{16} \mu^{k(1-\alpha)}
 \le \delta .
\end{align*} 

On the other hand, we see that $v_{k} \in S^{\ast} (\lambda/n, \Lambda,b_{F_{k}}, |f_{k}|+|g_{k}|+\mu^{2k}c_{\tilde{F}})$. Note that $b_{F_{k}} \le 1$ if $k$ is sufficiently large. 
Therefore by Lemma \ref{abptob}, 
\begin{align*}
& ||v_{k}||_{C^{0,\alpha_{0}}(Q_{1}^{\frac{\nu}{\mu^{k}}})} \\ &  \le  ||v_{k}||_{L^{\infty}(\partial_{p} Q_{1}^{\frac{\nu}{\mu^{k}}})}+ C(n, \lambda, \Lambda,  \delta_{0})( ||f_{k}||_{L^{n+1}(Q_{1}^{\frac{\nu}{\mu^{k}}})} \\ & \qquad  +||g_{k}||_{L^{n+1}(Q_{1}^{\frac{\nu}{\mu^{k}}})} +\mu^{2k}c_{\tilde{F}} +  \mu^{-k \alpha} ||\beta_{k} \cdot b_{k}||_{L^{\infty}(Q_{1} \cap \{ x_{n}=-\frac{\nu}{\mu^{k}} \})}) \\
& \le 1+ C(n, \lambda, \Lambda,  \delta_{0}) (  \delta + 1)
\\ & \le C(n, \lambda, \Lambda,  \delta_{0})
\end{align*}
for some $\alpha_{0}=\alpha_{0}(n, \lambda, \Lambda, \delta_{0}) \in (0,1)$. 
Note that we used $ ||v_{k}||_{L^{\infty}(Q_{1}^{\frac{\nu}{\mu^{k}}})} \le 1$ and $ ||\beta_{k}\cdot b_{k,s}||_{C^{0,\overline{\alpha}}(Q_{1} \cap \{ x_{n}=-\frac{\nu}{\mu^{k}} \})} \le \mu^{k\alpha}$ to obtain the last inequality. 

Now let $h \in C(\overline{Q}_{ \frac{7}{8}}^{\frac{\nu}{\mu^{k}}})$ be a solution of 
\begin{align} \label{appfh}
\left\{ \begin{array}{ll}
F_{k}(D^{2}h, 0, 0,0, 0) -h_{t} = 0 & \textrm{in $Q_{ \frac{7}{8}}^{\frac{\nu}{\mu^{k}}}$,}\\
h= v_{k}  & \textrm{on $ \partial_{p}Q_{ \frac{7}{8}}^{\frac{\nu}{\mu^{k}}}\cap \{ x_{n}>-\frac{\nu}{\mu^{k}} \}$,}\\
\beta_{k} \cdot Dh = -(\beta_{k} \cdot b_{k,s})/\mu^{k \alpha}   & \textrm{on $Q_{ \frac{7}{8}}\cap \{ x_{n}=-\frac{\nu}{\mu^{k}} \} $ }
\end{array} \right.
\end{align}
in the viscosity sense. 
Applying Lemma \ref{parobc1a} to $h$, we see that 
\begin{align} \label{po_addn}  ||h||_{C^{1,\overline{\alpha}}( \overline{Q}_{ \frac{3}{4}}^{\frac{\nu}{\mu^{k}}})} \le C_{2}(1+ \mu^{-k\alpha}||\beta_{k} \cdot b_{k,s}||_{C^{0,\overline{\alpha}}(Q_{ \frac{7}{8}}\cap \{ x_{n}=-\frac{\nu}{\mu^{k}} \})}) \le 2C_{2}
\end{align} for some $C_{2}=C_{2}(n, \lambda, \Lambda, \delta_{0} ,||\beta||_{C^{0,\overline{\alpha}}(Q_{2}^{\ast})})$.  
We also have
$$ ||v_{k}-h||_{L^{\infty}(Q_{\frac{3}{4}}^{\frac{\nu}{\mu^{k}}})} \le  \rho$$ 
by applying Lemma \ref{parobw1plem} to $v_{k}$ and $h$ with $\rho =4C_{2} \mu^{1+\overline{\alpha}} $.

Define $\overline{l}(x) = h(0,0)+ Dh(0,0)\cdot  x $.  
Then we can obtain 
\begin{align*}
 ||v_{k}-\overline{l} ||_{L^{\infty}(Q_{ 2\mu}^{+}  )} &  \le 
 ||v_{k}-h||_{L^{\infty}(Q_{ 2\mu}^{+})}+||h-\overline{l}||_{L^{\infty}(Q_{2\mu}^{+})} \\
 & \le 4C_{2} \mu^{1+\overline{\alpha}}+ \frac{1}{2}C_{2} (2\mu)^{1+\overline{\alpha}} \\
 & \le \mu^{1+ \alpha} 
\end{align*}
and this leads to 
$$ |\tilde{u}(x,t) -l_{k,s}(x) -\mu^{k(1+\alpha)} \overline{l} (\mu^{-k}x)| \le \mu^{(k+1)(1+\alpha)} $$
for any $ (x,t) \in Q_{ 2 \mu^{k+1}}^{\nu}$.
Therefore, we see that the first condition is satisfied if we set
$$l_{k+1,s}(x) = l_{k,s}(x)+ \mu^{k(1+\alpha)} \overline{l}(\mu^{-k}x).$$
Moreover, we can also check that  $a_{k+1,s}=a_{k,s}+ \mu^{k(1+\alpha)}h(0,0)$ and  $b_{k+1,s}=b_{k,s}+ \mu^{k\alpha}Dh(0,0)$ and this yields 
\begin{align*}
|a_{k,s}-a_{k+1,s}|+ \mu^{k}|b_{k,s}-b_{k+1,s}|& = \mu^{k(1+\alpha)}(|h(0,0)|+|Dh(0,0)|) 
\\ & \le \mu^{k(1+\alpha)} ||h||_{C^{1,\overline{\alpha}}(  Q_{ \frac{3}{4}}^{\frac{\nu}{\mu^{k}}}  )}
\\ & \le 2C_{2} \mu^{k(1+\alpha)}.
\end{align*}
Thus, the condition (ii) is also true.   

For (iv), we first observe that 
\begin{align*} (\beta_{k+1}\cdot b_{k+1,s})(x,t)=  \beta_{k}(\mu x,\mu^{2}t)\cdot (b_{k,s}+\mu^{k \alpha} Dh(0,0) )
\end{align*}
Since $\beta_{k} \cdot Dh = -(\beta_{k} \cdot b_{k,s})/\mu^{k \alpha}  $, we have
\begin{align*} (\beta_{k+1}\cdot b_{k+1,s})(x,t)=  \mu^{k\alpha}\beta_{k}(\mu x,\mu^{2}t)\cdot(Dh(0,0)-Dh(\mu x,\mu^{2}t)),
\end{align*}
and hence,
\begin{align*}&||\beta_{k+1}\cdot b_{k+1,s}||_{C^{0,\overline{\alpha}}(Q_{1}\cap \{ x_{n}=-\frac{\nu}{\mu^{k+1}} \})} \\& \le \mu^{k \alpha}
||\beta_{k}(\mu x,\mu^{2}t)\cdot(Dh(0,0)-Dh(\mu x,\mu^{2}t))||_{C^{0,\overline{\alpha}}(Q_{1}\cap \{ x_{n}=-\frac{\nu}{\mu^{k+1}} \})}.
\end{align*}
From \eqref{po_addn}, we can obtain
\begin{align*}
&||\beta_{k}(\mu x,\mu^{2}t)\cdot(Dh(0,0)-Dh(\mu x,\mu^{2}t))||_{C^{0,\overline{\alpha}}(Q_{1}\cap \{ x_{n}=-\frac{\nu}{\mu^{k+1}} \})}
\\ & \le (2\mu )^{\overline{\alpha}} \big\{ 2||h||_{C^{1,\overline{\alpha}}( \overline{Q}_{ \frac{3}{4}}^{\frac{\nu}{\mu^{k}}})}  + ||h||_{L^{\infty}( \overline{Q}_{ \frac{3}{4}}^{\frac{\nu}{\mu^{k}}})}||\beta ||_{C^{0,\overline{\alpha}}(Q_{1}^{\ast})}
 \big\}
 \\ & \le 4C_{2} (2+||\beta ||_{C^{0,\overline{\alpha}}(Q_{1}^{\ast})})\mu^{\overline{\alpha}}
 \\ & \le \mu^{\alpha}
\end{align*}
by using \eqref{mucond} and \eqref{po_addn}.
This yields that (iv) is true for $k+1$. 

Finally, we consider the condition (iii).
We first observe that
\begin{align} 
\left\{ \begin{array}{ll}
v_{k}-\tilde{l} \in S^{\ast}\big(\lambda/n, \Lambda, b_{k}, f_{k}+g_{k}+\frac{\delta}{8}\big) & \textrm{in $Q_{ 1}^{\frac{\nu}{\mu^{k}}}$,}\\
\beta_{k} \cdot D(v_{k}-\tilde{l}) = -\frac{\beta_{k} \cdot b_{k,s}}{\mu^{k \alpha}}-\beta \cdot Dh(0,0)   & \textrm{on $Q_{1}\cap \{ x_{n}=-\frac{\nu}{\mu^{k}} \} $}
\end{array} \right.
\end{align}
and
$$  -\frac{\beta_{k} \cdot b_{k,s}}{\mu^{k \alpha}}-\beta \cdot Dh(0,0) =\beta_{k}\cdot (Dh-Dh(0,0))$$
from the definition of $v_{k}$ and $\overline{l}$.
Then, by using Lemma \ref{parobhol}, we have
\begin{align*}
&||v_{k}-\tilde{l}||_{C^{0,\alpha_{0}}(Q_{\nu}^{\frac{\nu}{\mu^{k}}})} \\ & \le
C_{1} \mu^{-\alpha_{0}}(\mu^{1+\alpha}+2\delta\mu^{2-\frac{n+2}{p'}} +\mu ||\beta_{k}\cdot (Dh-Dh(0,0))||_{L^{\infty}(Q_{\mu}\cap \{ x_{n}=-\frac{\nu}{\mu^{k}} \} )} ).
\end{align*}
Now we choose sufficiently small $\delta$ with $2\delta \le \mu^{\alpha+\frac{n+2}{p'}-1}$.
We can also check that 
$$||\beta_{k}\cdot (Dh-Dh(0,0))||_{L^{\infty}(Q_{\mu}\cap \{ x_{n}=-\frac{\nu}{\mu^{k}} \} )}
\le \mu^{\overline{\alpha}}||h||_{C^{1,\overline{\alpha}}( \overline{Q}_{ \frac{3}{4}}^{\frac{\nu}{\mu^{k}}})} \le 2C_{2}\mu^{\overline{\alpha}}.$$
This leads to 
$$ ||v_{k}-\tilde{l}||_{C^{0,\alpha_{0}}(Q_{\nu}^{\frac{\nu}{\mu^{k}}})} \le C_{1} \mu^{-\alpha_{0}}(2\mu^{1+\alpha}+ 2C_{2}\mu^{1+\overline{\alpha}})\le 3C_{1}\mu^{1+\alpha-\alpha_{0}}. $$
Therefore, we can obtain
\begin{align*}
&|(\tilde{u}-l_{k+1,s})(\mu^{k+1}x_{1},\mu^{2(k+1)}t_{1})-(\tilde{u}-l_{k+1,s})(\mu^{k+1}x_{2},\mu^{2(k+1)}t_{2})|
\\ & \le 3C_{1} \mu^{(k+1)(1+\alpha)} (|x_{1}-x_{2}|+|t_{1}-t_{2}|^{\frac{1}{2}})^{\alpha_{0}}
\end{align*}
for any $(x_{1},t_{1}),(x_{2},t_{2})\in Q_{1}\cap \{x_{n}\ge -\nu \}$.

Hence, we can always find a linear function $l_{s}$ and a universal constant $C$ with
\begin{align} \label{pafdest}  |l_{s}(0,0)|,|Dl_{s}(0,0)| \le CK(y,s) 
\end{align} and
\begin{align} \label{palest}  ||u-l_{s}||_{L^{\infty}(Q_{r}(y,s)\cap Q_{1}^{+})} \le  Cr^{1+\alpha} K(y,s) 
\end{align}
for any $ (y,s) \in Q_{\frac{1}{2}}\cap \{ x_{n} \ge 0 \}$ with $y_{n} < \sigma /2 $ and sufficiently small $r>0$.

Next, in the case  $y_{n} \ge \sigma /2 $, we can refer to the interior $C^{1,\alpha}$-regularity 
in \cite[Lemma 7.4]{MR1789919}.
Thus, we get the estimate \eqref{pake1} in the case $ p > n+2$ with  $p'=n+2$ and $\alpha< 1-\frac{n+2}{p} $ since
\begin{align*}
K(y,s)\le ||u||_{L^{\infty}(Q_{d}(y,s) \cap Q_{1}^{+})}  +\epsilon_{0}^{-1}\sup_{r\le d}\big( r^{1+\alpha-\frac{n+2}{p}}||f||_{L^{p}(Q_{1}^{+})} \big).
\end{align*}

Besides, 
we also see that  \eqref{pafdest} and \eqref{palest} are also satisfied for almost every $(y,s) \in Q_{\frac{1}{2}}^{+}$ when $ n+1   < p \le n+2 $. 
Then we get
$$ \frac{|u(y+x,s+t)-u(y,s)|}{|x|+|t|^{\frac{1}{2}}} \le C  K(y,s)$$ for some $C>0$ and almost every $(y,s) \in Q_{\frac{1}{2}}^{+}$ and $(x,t) \in Q_{r} \backslash \{ (0,0) \}$ such that $(y+x,s+t) \in   Q_{1}^{+}$.
Write $$ I_{(x,t)}(y,s) := \frac{|u(y+x,s+t)-u(y,s)|}{|x|+|t|^{\frac{1}{2}}} .$$
It is straightforward to check that 
\begin{align*}
 ||I_{(x,t)}||_{L^{q}(Q_{\frac{1}{2}}^{+})}   \le C ||K(\cdot, \cdot)||_{L^{q}(Q_{\frac{1}{2}}^{+})}
 \le C(||u||_{L^{\infty}(Q_{1}^{+})}  + J)
\end{align*}
for any $q \in (p', (n+2)p'/[n+2-p'(1-\alpha)]) $, 
where
$$ J =  \bigg\{ \int_{Q_{\frac{1}{2}}^{+}} \hspace{-0.3em} \sup_{r \le d}  \bigg[r^{q(1-\alpha)} \bigg(  \hspace{-0.25em} r^{-(n+2)} \hspace{-0.5em} \int_{Q_{r} (y,s)\cap  Q_{1}^{+}} \hspace{-0.7em}    | f( x,t)|^{p'} \ dx dt \bigg)^{\frac{q}{p'}} \bigg] dy ds \bigg\}^{\frac{1}{q}} .$$
As in the proof of \cite[Lemma 7.4]{MR1789919}, we obtain $ J \le C||f||_{L^{p}(Q_{1}^{+})} $ for some  $C=C(n,p,p')$,
and then
\begin{align*}
\sup_{|(x,t)|<r}||I_{(x,t)}||_{L^{q}(Q_{\frac{1}{2}}^{+})} \le C (  ||u||_{L^{\infty}(Q_{1}^{+})}+||f||_{L^{p}(Q_{1}^{+})})
\end{align*}
for any $p' \le q < p^{\ast} $ with proper $ p' $ and $ \alpha$.
Then we can get the desired result.
\end{proof}

Thanks to Theorem \ref{paobw1p}, we get $W^{2,p}$-estimate for viscosity solutions of \eqref{paob_inq1} directly.
\begin{corollary}  \label{paob_flat}
Let $n+1 <p <\infty$ and $u $ be a viscosity solution of \eqref{paob_inq1}.
Then, under the assumption of Theorem \ref{paobw1p}, $ u \in W^{2,p}(Q_{1/4}^{+})$ and
$$ ||u||_{W^{2,p}(Q_{1/4}^{+})} \le C (||u||_{L^{\infty}(Q_{1}^{+})}+ ||f||_{L^{p}(Q_{1}^{+})})$$
for some $C= C(n, \lambda, \Lambda , b,c, p, ||\beta||_{C^{2}(\overline{Q}_{1}^{\ast})} )$.
\end{corollary}

\subsection{Global $W^{2,p}$-regularity for \eqref{ob_eqoripa}} 
We can establish the global $W^{2,p}$-regularity for \eqref{ob_eqoripa}  by using Corollary \ref{paob_flat}. 
\begin{theorem} \label{paob_main}
Let $ \Omega$ be a bounded $C^{3} $-domain with $ T>0$ and
$\mathbf{n}$ be the inward unit normal vector to $\partial \Omega$. 
Assume that $u$ is a viscosity solution of \eqref{ob_eqoripa},
 where $F(X,q,r,x,t)$ is convex in $X$, continuous in $x$ and $t$, and satisfies the structure condition \eqref{paob_sc} with $F(0,0,0,x,t)=0 $, 
$\beta \in C^{2}(\partial \Omega \times (0,T)) $ with $\beta \cdot \mathbf{n} \ge \delta_{0}$ for some $\delta_{0} > 0$ and $f \in L^{p}(  \Omega \times (0,T)) \cap C(  \Omega \times (0,T))$ for $ n +1 < p < \infty $.
Then there exists $\epsilon_{0}$ depending on $ n,  p,  \lambda, \Lambda, \delta_{0}$ and $ ||\beta||_{C^{2}(\partial \Omega \times (0,T))}$ such that if
$$\bigg(  \kint_{Q_{r}(x_{0},t_{0}) \cap (\Omega \times (0,T])}     \psi_{F}( (x_{0},t_{0}),(x,t))^{p} \ dx dt \bigg)^{1/p} \le \epsilon_{0} $$
for  any $(x_{0} ,t_{0})\in \Omega\times (0,T]$ and $0<r<r_{0}$, then $ u \in W^{2,p}( \Omega\times (0,T] )$ with the global $W^{2,p}$-estimate
$$ ||u||_{ W^{2,p}(\Omega\times (0,T))} \le C (||u||_{ L^{\infty}(\Omega\times (0,T) )}+||f||_{ L^{p}( \Omega\times (0,T) )})$$
for some $C$ depending only on $ n,  p,  \lambda, \Lambda,  \delta_{0},   b, c,r_{0}, ||\beta||_{C^{2}(\partial \Omega \times (0,T))},T$ and $ \diam(\Omega)$.
\end{theorem}
\begin{proof}
First, we get the following interior $W^{2,p}$-estimate
$$ ||u||_{ W^{2,p}(Q)} \le C (||u||_{ L^{\infty}(Q )}+||f||_{ L^{p}( Q )}) $$ for any $Q \subset \subset \Omega \times (0,T)$
from \cite[Theorem 5.9]{MR1135923}.
Thus, it is sufficient to consider the boundary case.

We are going to use a flattening argument in order to get a boundary estimate.  
For any $x_{0} \in \partial \Omega$, we can find a neighborhood $N(x_{0}) $ of $x_{0}$ and a $ C^{3}$-diffeomorphism $ \Psi : U(x_{0}) \to B_{1}^{+} $
with $ \Psi(x_{0}) = 0$ since $ \partial \Omega $ is $ C^{3} $. 
Then for each $t_{0} \in (0,T]$, we define $\Psi_{t_{0}} : U(x_{0}) \times (t_{0}-1, t_{0}) \to Q_{1}^{+}  $ such that
$$\Psi_{t_{0}}(x,t) = (\Psi (x), t-t_{0}). $$

Fix $t_{0} \in (0,T]$ and let  $ \tilde{u} =  u \circ \Psi_{t_{0}}^{-1}$. 
Then $\tilde{u} $ is a solution of  
 \begin{align*} 
\left\{ \begin{array}{ll}
\tilde{F}(D^{2}\tilde{u},D\tilde{u},\tilde{u},x,t) - \tilde{u}_{t}=\tilde{f} & \textrm{in $Q_{1}^{+}$,}\\
\tilde{\beta} \cdot D\tilde{u} = 0 & \textrm{on $Q_{1}^{\ast}$}\\
\end{array} \right. 
\end{align*}
in the viscosity sense,
where $ \tilde{f} = f \circ \Psi_{t_{0}}^{-1}$, 
$\tilde{\beta} = ( \beta \circ \Psi_{t_{0}}^{-1}) \cdot (D\Psi_{t_{0}} \circ \Psi_{t_{0}}^{-1})^{t}$ and 
\begin{align*}
\tilde{F}(D^{2}\tilde{\varphi},&D\tilde{\varphi},\tilde{u},x,t) 
=F (D^{2}\varphi, D\varphi, u ,x,t) \circ \Psi_{t_{0}}^{-1} \\ &
= F(D \Psi_{t_{0}}^{T} \circ \Psi_{t_{0}}^{-1} D^{2} \tilde{\varphi} D\Psi_{t_{0}} \circ \Psi_{t_{0}}^{-1} + (D \tilde{\varphi}  \partial_{i,j} \Psi_{t_{0}} \circ \Psi_{t_{0}}^{-1})_{1\le i,j \le n} , \\ & \qquad  \qquad
D\varphi D \Psi_{t_{0}} \circ \Psi_{t_{0}}^{-1}, \tilde{u} , \Psi_{t_{0}}^{-1}(x,t))
\end{align*}
for $ \tilde{\varphi} \in W^{2,p}(Q_{1}^{+} )$ and $ \varphi = \tilde{\varphi} \circ \Psi_{t_{0}}   \in W^{2,p}(U(x_{0})\times (t_{0}-1,t_{0})) $.
Here, we also note that we extended $u$ by zero when $t<0$.

Now we can see that there exists a uniform constant $C(\Psi)$ such that
$$ \psi_{\tilde{F}}((x,t),(x_{0},t_{0})) \le C(\Psi) \psi_{F}((\Psi^{-1}(x,t)),(\Psi^{-1}(x_{0},t_{0})))$$ and $\tilde{F}$ is uniformly elliptic with constants $\lambda  C(\Psi)$, $ \Lambda  C(\Psi)$, see \cite{MR2486925}.
In addition, we also have $ \tilde{\beta}= ( \beta \circ \Psi_{t_{0}}^{-1}) \cdot (D\Psi_{t_{0}} \circ \Psi_{t_{0}}^{-1})^{t} \in C^{2}$ since $\Psi, \Psi^{-1} \in C^{3}$ and $\beta\in C^{2}(\partial \Omega \times (0,T)) $.
Therefore, we can obtain  the boundary estimate, thanks to Corollary \ref{paob_flat} along with a standard covering argument.
This completes the proof.
\end{proof}

\end{document}